\documentclass[12pt,reqno,twoside]{amsart}
\usepackage{etex}
\usepackage{amssymb,amsmath,amscd,enumerate,verbatim,extarrows}
\usepackage[colorlinks=true,linkcolor=blue,citecolor=blue]{hyperref}
\usepackage[latin1]{inputenc}
\usepackage{tikz}
\usetikzlibrary{positioning,chains,fit,shapes,calc,shapes.geometric}
\usepackage{color,pst-node,multido}

\usepackage[all]{xy}
\usepackage{pstricks, pst-3d}
\newcommand{\mm}{\mathfrak m}
\newcommand{\nn}{\mathfrak n}
\newcommand{\pp}{\mathfrak p}

\newcommand{\Z}{\mathbb{Z}}

\newcommand{\N}{\mathbb{N}}

\newcommand{\Fc}{\mathcal{F}}
\newcommand{\Gc}{\mathcal{G}}

\DeclareMathOperator{\pnt}{\raise 0.5mm \hbox{\large\bf.}}

\DeclareMathOperator{\Coker}{Coker}

\DeclareMathOperator{\Img}{Im}

\DeclareMathOperator{\inmat}{indmatch}

\DeclareMathOperator{\gr}{gr}

\DeclareMathOperator{\chara}{char}
\DeclareMathOperator{\Tor}{Tor}

\DeclareMathOperator{\lind}{ld}

\DeclareMathOperator{\linp}{lin}

\DeclareMathOperator{\reg}{reg}

\DeclareMathOperator{\projdim}{pd}

\DeclareMathOperator{\supp}{supp}

\def\+#1{\relax\ifmmode\if\noexpand #1\relax \mathop{\kern
    0pt^+{#1}}\nolimits\else \kern 0pt^+\!#1 \fi\else$^*$#1\fi}

\newcommand{\shift}{\mathsf{\Sigma}}
\newcommand{\vphi}{\varphi}

\let\:=\colon

\newtheorem{thm}{\bf Theorem}[section]
\newtheorem{lem}[thm]{\bf Lemma}
\newtheorem{cor}[thm]{\bf Corollary}
\newtheorem{prop}[thm]{\bf Proposition}

\newtheorem{quest}[thm]{\bf Question}

\theoremstyle{definition}
\newtheorem{defn}[thm]{\bf Definition}

\theoremstyle{plain}
\newtheorem*{thm*}{Theorem}
\newtheorem*{lem*}{Lemma}
\newtheorem*{cor*}{Corollary}
\newtheorem*{claim*}{Claim}
\newtheorem*{defn*}{Definition}

\theoremstyle{remark}
\newtheorem{rem}[thm]{Remark}

\newtheorem{ex}[thm]{Example}

\numberwithin{equation}{section}
\title[Linearity defect of edge ideals]{Linearity defect of edge ideals and Fr\"oberg's theorem}

\author{Hop D. Nguyen}
\address{Fachbereich Mathematik/Informatik, Institut f\"ur Mathematik, Universit\"at Osnabr\"uck, Albrechtstr. 28a, 49069 Osnabr\"uck, Germany}
\address{Dipartimento di Matematica, Universit\`a di Genova, Via Dodecaneso 35, 16146 Genoa, Italy}
\email{ngdhop@gmail.com}
\author{Thanh Vu}
\address{Department of Mathematics, University of Nebraska-Lincoln, Lincoln, NE 68588, USA}
\email{tvu@unl.edu}
\thanks{This work is partially supported by the NSF grant DMS-1103176. Part of this work was done when Nguyen was a Marie Curie fellow of the Istituto Nazionale di Alta Matematica.}


\begin{document}

\begin{abstract}
Fr\"oberg's classical theorem about edge ideals with $2$-linear resolution can be regarded as a classification of graphs whose edge ideals have linearity defect zero. Extending his theorem, we classify all graphs whose edge ideals have linearity defect at most $1$. Our characterization is independent of the characteristic of the base field: the graphs in question are exactly weakly chordal graphs with induced matching number at most $2$. The proof uses the theory of Betti splittings of monomial ideals due to Francisco, H\`a, and Van Tuyl and the structure of weakly chordal graphs. Along the way, we compute the linearity defect of edge ideals of cycles and weakly chordal graphs. We are also able to recover and generalize previous results due to Dochtermann-Engstr\"om, Kimura and Woodroofe on the projective dimension and Castelnuovo-Mumford regularity of edge ideals.
\end{abstract}

\maketitle
\section{Introduction}

Let $(R,\mm)$ be a standard graded algebra over a field $k$ with the graded maximal ideal $\mm$. Let $M$ be a finitely generated graded $R$-module. For integers $i,j$, the $(i,j)$-graded Betti number of $M$ is defined by $\beta_{i,j}(M)=\dim_k \Tor^R_i(k,M)_j$. The following number is called the {\it (Castelnuovo-Mumford) regularity} of $M$ over $R$
\[
\reg_R M =\sup\{j-i: \beta_{i,j}(M)\neq 0\}.
\]
The regularity is an important invariant of graded modules over $R$. When $R$ is a polynomial ring and $M$ is a monomial ideal of $R$, its regularity exposes many combinatorial flavors. This fact has been exploited and proved to be very useful for studying the regularity; for recent surveys, see \cite{Ha}, \cite{MV}, \cite{V}. A classical and instructive example is Fr\"oberg's theorem. Recall that if $G$ is a graph on the vertex set $\{x_1,\ldots,x_n\}$ (where $n\ge 1$), and by abuse of notation, $R$ is the polynomial ring $k[x_1,\ldots,x_n]$, then the {\it edge ideal} of $G$ is $I(G)=(x_ix_j: \{x_i,x_j\} ~ \text{is an edge of $G$})$. Unless otherwise stated, whenever we talk about an invariant of $I(G)$ (including the regularity and the linearity defect, to be defined below), it is understood that the base ring is the polynomial ring $R$. For $m\ge 3$, the {\it cycle} $C_m$ is the graph on vertices $x_1,\ldots,x_m$ with edges $x_1x_2,\ldots,x_{m-1}x_m,x_mx_1$. We say that a graph $G$ is {\it weakly chordal} (or weakly triangulated) if for every $m\ge 5$, neither $G$, nor its complement contains $C_m$ as an induced subgraph. $G$ is {\it chordal} if for any $m\ge 4$, $C_m$ is not an induced subgraph of $G$. Fr\"oberg's theorem \cite{Fr} says that $I(G)$ has regularity $2$ if and only if the complement graph of $G$ is chordal. It is of interest to find generalizations of this important result; see, for instance, \cite{EGHP}, \cite{FG}. Fern\'andez-Ramos and Gimenez \cite[Theorem 4.1]{FG} extended Fr\"oberg's theorem by providing a combinatorial characterization of connected bipartite graphs $G$ such that $\reg I(G)=3$. In general, the regularity $3$ condition on edge ideals is not purely combinatorial as it depends on the characteristic of the field; see Katzman's example in \cite[Page 450]{K}.

The linearity defect was introduced by Herzog and Iyengar \cite{HIy} motivated by work of Eisenbud, Fl{\o}ystad, and Schreyer \cite{EFS}. It is defined via the linear part of minimal free resolutions of modules over $R$, see Section \ref{sect_background} for details. The linear part appears naturally: from \cite[Theorem 3.4]{EFS}, taking homology of a complex over a polynomial ring is equivalent to taking the linear part of a minimal free complex over the (Koszul dual) exterior algebra. The linearity defect itself is interesting because it yields stronger homological information than the regularity: over a local ring, Herzog and Iyengar \cite[Proposition 1.8]{HIy} proved that modules with finite linearity defect have rational Poincar\'e series with constant denominator. On the other hand, there exist modules which have finite regularity and transcendental Poincar\'e series \cite[Page 252]{J}. Furthermore, the linearity defect is flexible enough to generate a rich theory: going beyond regular rings, one still encounters reasonable classes of rings over which every module has finite linearity defect, e.g. exterior algebras \cite{EFS}, homogeneous complete intersections defined by quadrics \cite{HIy}.

Let $R$ again be a polynomial ring over $k$ and $M$ a finitely generated graded $R$-module. The linearity defect of $M$ over $R$ is denoted by $\lind_R M$ or simply $\lind M$. Then $\lind M$ equals $\ell$ (where $\ell \ge 0$) if and only if the first syzygy module of $M$ which is componentwise linear in the sense of Herzog and Hibi \cite{HH} is the $\ell$-th one. In particular, the condition that an edge ideal has $2$-linear resolution is equivalent to the condition that its linearity defect is $0$. Fr\"oberg's theorem can be rephrased as a classification of edge ideals with linearity defect $0$. Our motivation is to find a purely combinatorial characterization for linearity defect $1$ edge ideals; it turns out that indeed there is one. Finding such a characterization is non-trivial due to several reasons. First, while the Hochster's formula can give much information about the regularity of Stanley-Reisner ideals (see for example Dochtermann and Engstr\"om \cite{DE}), up to now there is no combinatorial interpretation of the linearity defect of Stanley-Reisner ideals. (See \cite{OkaYan}, \cite{Ro} for some results about the linearity defect of such ideals.) Second and furthermore, the linearity defect generally cannot be read off from the Betti table: for example (see \cite[Example 2.8]{HSV}), the ideals 
$$
I_1=(x_1^4,x_1^3x_2,x_1^2x_2^2,x_1x_2^3,x_2^4,x_1^3x_3,x_1^2x_2x_3^2,x_1^2x_3^3,x_1x_2^2x_3^2)
$$ 
and 
$$
I_2=(x_1^4,x_1^3x_2,x_1^2x_2^2,x_1^3x_3,x_1x_2^2x_3,x_1x_2x_3^2,x_1x_2^4,x_1^2x_3^3,x_2^4x_3)
$$ in $k[x_1,x_2,x_3]$ have the same graded Betti numbers, but the first one has linearity defect $0$ while the second one has positive linearity defect (equal to $1$). Third, the quest of finding the aforementioned characterization yields interesting new insights even to the more classical topics concerning Castelnuovo-Mumford regularity or the projective dimension. Indeed, in proving one of the main results (Theorem \ref{thm_weaklychordal}), we recover a theorem of Woodroofe \cite[Theorem 14]{W} on regularity of edge ideals of weakly chordal graphs. In Theorem \ref{thm_projdim_weaklychordal}, we prove a new result about the projective dimension of edge ideals of weakly chordal graphs, extending previous work of Dochtermann and Engstr\"om \cite{DE} and Kimura \cite{Ki1}, \cite{Ki2}.

As mentioned earlier, our motivation is to see if the linearity defect one condition is purely combinatorial. Recall that for $g\ge 1$, the $gK_2$ graph is the graph consisting of $g$ disjoint edges. The {\it induced matching number} of a graph $G$, denoted by $\inmat(G)$, is the largest number $g$ such that there is an induced $gK_2$ subgraph in $G$. The first main result of our paper is the following new generalization of Fr\"oberg's theorem.
\begin{thm}[See Theorem \ref{thm_ld1}]
\label{thm_main1}
Let $G$ be a graph. Then $\lind I(G)=1$ if and only if $G$ is a weakly chordal graph with induced matching number $\inmat(G)=2$.
\end{thm}
This suggests (to us) a surprising, if little exploited connection between the linearity defect of edge ideals and combinatorics of graphs. The hard part of Theorem \ref{thm_main1}, the sufficiency, follows from a more general statement about the linearity defect of edge ideals of weakly chordal graphs. Thus the second main result of our paper, which was inspired by the aforementioned theorem of Woodroofe \cite[Theorem 14]{W}, is
\begin{thm}[See Theorem \ref{thm_weaklychordal}]
\label{thm_main2}
Let $G$ be a weakly chordal graph with at least one edge. Then there is an equality $\lind I(G)=\inmat(G)-1.$
\end{thm}
Our proof of Theorem \ref{thm_main2} takes a cue from the theory of Betti splittings due to Francisco, H\`a, and Van Tuyl \cite{FHV}. Let $I$ be a monomial ideal of $R$, and $\Gc(I)$ its set of uniquely determined minimal monomial generators. Let $J, K$ be monomial ideals contained in $I$ such that $\Gc(J)\cap \Gc(K)=\emptyset$ and $\Gc(I)=\Gc(J)\cup \Gc(K)$, so that in particular $I=J+K$. The decomposition of $I$ as $J+K$ is called a {\em Betti splitting} if for all $i\ge 0$ and all $j\ge 0$, the following equality of Betti numbers
\begin{equation}
\label{eq_Betti}
\beta_{i,j}(I)=\beta_{i,j}(J)+\beta_{i,j}(K)+\beta_{i-1,j}(J\cap K) 
\end{equation}
holds. In \cite{FHV}, \cite{HV1}, \cite{HV2}, Betti splittings were used to study Betti numbers and regularity of edge ideals and more general squarefree monomial ideals. What makes Betti splittings useful to the study of linearity defect is the following fact, proven in \cite[Proposition 2.1]{FHV}: the decomposition $I=J+K$ is a Betti splitting if and only if the natural maps $\Tor^R_i(k, J\cap K) \longrightarrow \Tor^R_i(k, J)$ and $\Tor^R_i(k, J\cap K) \longrightarrow \Tor^R_i(k, K)$ are trivial for all $i\ge 0$. It is proved in Proposition \ref{prop_zeromap} of Section \ref{sect_Betti_split} that we have a good control of the linearity defect along short exact sequences for which certain induced maps of $\Tor$ have strong vanishing properties. This result implies that Betti splittings are suitable for bounding the linearity defect (Theorem \ref{thm_Betti_splitting}). 

The second component of the proof of Theorem \ref{thm_main2} comes from the structure theory of weakly chordal graphs. Specifically, we use the existence of {\it co-two-pair} edges \cite{HHM} in a weakly chordal graph. The main work of Section \ref{sect_weaklychordal} is to show that any co-two-pair in a weakly chordal graph gives rise to a Betti splitting of the corresponding edge ideal; see Theorem \ref{thm_weaklychordal}. A variety of techniques is employed to prove the last result, including the theory of lcm-lattice in monomial resolutions developed in \cite{GPW} and \cite{PV}.

The paper is organized as follows. We start by recalling the necessary background in Section \ref{sect_background}. Section \ref{sect_induced_matchings} provides a lower bound for the linearity defect of edge ideals in terms of the induced matching number of the associated graphs. This bound plays a role in the proof of the necessity part of Theorem \ref{thm_main1} and in Theorem \ref{thm_main2}. The main result of Section \ref{sect_Betti_split} is that linearity defect behaves well with respect to Betti splittings (Theorem \ref{thm_Betti_splitting}). In Section \ref{sect_weaklychordal}, we compute the linearity defect of edge ideals associated to weakly chordal graphs. In particular we prove in Section \ref{sect_weaklychordal} Theorem \ref{thm_main2} introduced above. Section \ref{sect_cycles} concerns with the computation of linearity defect for the simplest non-weakly-chordal graphs, namely cycles of length $\ge 5$. (For complements of cycles, the computation was done in \cite[Theorem 5.1]{OkaYan}.)  Besides applications to the theory of regularity and projective dimension of edge ideals, in Section \ref{sect_applications}, we prove Theorem \ref{thm_main1}. In Section \ref{sect_dependence_characteristic}, we study the dependence of the linearity defect of edge ideals on the characteristic of the field $k$, and propose some open questions. 

Since the linearity defect was originally defined in \cite{HIy} for modules over local rings, we state some of our results in this greater generality; see for example Proposition \ref{prop_zeromap}, Theorem \ref{thm_Betti_splitting} and Proposition \ref{prop_Koszul_decomp}. The reader may check easily that the analogues of these results for graded algebras are also true, using the same method.
\section{Background}
\label{sect_background}
We assume that the reader is familiar with the basic of commutative algebra; a good reference for which is \cite{BH}. For the theory of free resolutions, we refer to \cite{Avr}.
\subsection{Linearity defect}
Let $(R,\mm,k)$ be a standard graded $k$-algebra with the graded maximal ideal $\mm$, or a noetherian local ring with the maximal ideal $\mm$ and the residue field $k$. By a ``standard graded $k$-algebra'', we mean that $R$ is a commutative algebra over $k$, $R$ is $\N$-graded with $R_0=k$, and $R$ is generated over $k$ by finitely many elements of degree $1$. Sometimes, we omit $k$ and write $(R,\mm)$ for simplicity.

Let us define the linearity defect for complexes of modules over local rings; the modification for graded algebras is straightforward. Let $(R,\mm)$ be a noetherian local ring, and $M$ be a chain complex of $R$-modules with homology $H(M)$ degreewise finitely generated and bounded below, i.e. $H_i(M)$ is finitely generated for all $i\in \Z$ and $H_i(M)=0$ for $i\ll 0$. Let $F$ be its minimal free resolution:
\[
F: \cdots \longrightarrow  F_i \longrightarrow F_{i-1} \longrightarrow \cdots \longrightarrow F_1 \longrightarrow F_0 \longrightarrow F_{-1} \longrightarrow \cdots.
\]
In particular, up to isomorphism of complexes, $F$ is the unique complex of finitely generated free $R$-modules that fulfills the following conditions:
\begin{enumerate}
\item $\Img(F_i) \subseteq \mm F_{i-1}$ for all $i\in \Z$,
\item there is a morphism of complexes $F\longrightarrow M$ which induces isomorphism on homology.
\end{enumerate}
The complex $F$ can be chosen such that $F_i=0$ for all $i<\inf M:=\inf\{i: H_i(M)\neq 0\}$. See the monograph of Roberts \cite{R} for more details.

The complex $F$ admits a filtration $(\Fc^iF)_{i\ge 0}$, where $\Fc^iF$ is the complex
\[
\Fc^iF: \cdots \longrightarrow F_{i+1} \longrightarrow  F_i \longrightarrow \mm F_{i-1} \longrightarrow \cdots \longrightarrow \mm^{i-1}F_1 \longrightarrow \mm^i F_0 \longrightarrow \mm^{i+1}F_{-1} \longrightarrow \cdots,
\]
with the differential being induced by that of $F$. The complex 
$$
\linp^R F=\bigoplus_{i\in \Z} \frac{\Fc^iF}{\Fc^{i+1}F}
$$
is called the {\em linear part} of $F$. It is a complex of graded free $\gr_{\mm}R$-modules. Here, as usual, 
\[
\gr_{\mm}R=\bigoplus_{i\ge 0} \frac{\mm^i R}{\mm^{i+1}R}
\]
is the associated graded ring of $R$ with respect to the $\mm$-adic filtration. By a straightforward computation, one has for all $i \ge 0$ an isomorphism of graded $\gr_{\mm}R$-modules:
\begin{equation}
\label{eq_linp}
 (\linp^R F)_i=\left(\gr_{\mm}F_i\right)(-i) \cong \frac{F_i}{\mm F_i} \otimes_{R/\mm}(\gr_{\mm}R)(-i).
\end{equation}
It is worth pointing out here a simple procedure for computing the linear part of minimal free resolutions if $R$ is a graded algebra. Now $M$ is a complex of graded $R$-modules with $H(M)$ degreewise finitely generated and bounded below, and $F$ is the minimal graded free resolution of $M$. It is not hard to see that $\linp^R F$ has the same underlying module structure as $F$ itself, and the matrices of differentials of $\linp^R F$ are obtained from that of $F$ by replacing each non-zero entry of degree at least $2$ by zero. 

For example, let $R=k[x,y]$ and $I=(x^2,xy^2,y^4)$, then a minimal graded free resolution of $I$ is
\[
F: \, 0\longrightarrow 
R(-4) \oplus R(-5) \xrightarrow{\left(
	\begin{array}{c c c}
	-y^2 & 0\\
	x & - y^2\\   	
	0  & x\\
	\end{array}\right)} R(-2)\oplus R(-3) \oplus R(-4) \longrightarrow 0.
\]
The linear part of $F$ is the following complex
\[
\linp^R F:\, 0\longrightarrow 
R(-1)^2 \xrightarrow{\left(
	\begin{array}{c c c}
	0 & 0\\
	x & 0\\   	
	0  & x\\
	\end{array}\right)} R^3 \longrightarrow 0.
\]

The {\it linearity defect} of $M$ over $R$, denoted by $\lind_R M$, is defined as follows:
\[
\lind_R M =\sup\{i: H_i(\linp^R F) \neq 0\}.
\]
By convention, $\lind_R M=0$ if $M$ is the trivial module $0$. Except for the proof of Proposition \ref{prop_zeromap}, we will work solely with linearity defect of modules. The notion of linearity defect was introduced by Herzog and Iyengar in 2005; see their paper \cite{HIy} for more information about the homological significance of the complex $\linp^R F$ and the linearity defect.

In the above example, we can verify that $H_1(\linp^R F)=0$ and $\lind_R I=0$. Following \cite{HIy}, modules which have linearity defect zero are called {\em Koszul modules}.

\subsection{Castelnuovo-Mumford regularity} 
\label{sect_regularity}
Let $(R,\mm)$ now be a standard graded $k$-algebra, and $M$ a finitely generated graded $R$-module. The Castelnuovo-Mumford regularity of $M$ over $R$ is 
\[
\reg_R M=\sup\{j-i: \Tor^R_i(k,M)_j\neq 0\}.
\]
Ahangari Maleki and Rossi \cite[Proposition 3.5]{AR} showed that if $\lind_R M=\ell <\infty$ then the regularity of $M$ can be computed using the first $\ell$ steps in its minimal free resolution:
\[
\reg_R M=\sup\{j-i: \Tor^R_i(k,M)_j\neq 0 ~\text{and $i\le \ell$} \}.
\]
In particular, if $M$ is a Koszul $R$-module then $\reg_R M$ equals the maximal degree of a minimal homogeneous generator of $M$.

We say that $M$ has a linear resolution if for some $d\in \Z$, $M$ is generated in degree $d$ and $\reg_R M=d$. We also say that $M$ has a $d$-linear resolution in that case.

R\"omer \cite{Ro} proved that if $R$ is a Koszul algebra, i.e. $\reg_R k=0$, then the following statements are equivalent:
\begin{enumerate}
\item $M$ is a Koszul module;
\item $M$ is componentwise linear, namely for every $d\in \Z$, the submodule $M_{\left<d\right>}$ of $M$ generated by homogeneous elements of degree $d$ has a $d$-linear resolution.
\end{enumerate}
A proof of this result can be found in \cite[Theorem 5.6]{IyR}. From R\"omer's theorem, one gets immediately that if $R$ is a Koszul algebra, and $M$ is generated in a single degree $d$, then $\lind_R M=0$ if and only if $M$ has a $d$-linear resolution.

\subsection{Graphs and their edge ideals}
We always mean by a graph $G$ a pair $(V(G), E(G))$, where $V(G)=\{x_1,\ldots,x_n\}$ is a finite set (with $n\ge 1$), and $E(G)$ is a collection of non-ordered pairs $\{x_i,x_j\}$ where $1\le i, j\le n, i\neq j$. Elements in $V(G)$ are called {\it vertices} of $G$, while elements in $E(G)$ are termed {\it edges} of $G$. In this paper, we do not consider infinite graphs, nor do we consider graphs with loops or multiple edges between two vertices. Most of the material on graph theory that we need can be found in \cite{V}. 

For a graph $G$ and $x\in V(G)$, a vertex $y\in V(G)$ is called a {\it neighbor} of $x$ if $\{x,y\} \in E(G)$. We also say that $y$ is {\it adjacent} to $x$ in that case. We denote by $N(x)$ the set of neighbors of $x$. 

The graph $G^c$ with the same vertices as $G$ and with edge set consisting of non-ordered pairs $\{x,y\}$ of non-adjacent vertices of $G$, is called the {\it complement graph} of $G$. 

By a {\it cycle}, we mean a graph with vertices $x_1,\ldots,x_m$ (with $m\ge 3$) and edges $x_1x_2,\ldots,$ $x_{m-1}x_m$, $x_mx_1$. In that case, $m$ is called the length of the cycle. The complement graph of a cycle (of length $m$) is called an {\it anticycle} (of length $m$). The path of length $m-1$ (where $m\ge 2$), denoted by $P_m$, is the graph with vertices $x_1,\ldots,x_m$ and edges $x_1x_2,\ldots,x_{m-1}x_m$. 

A subgraph of $G$ is a graph $H$ such that $V(H)\subseteq V(G)$ and $E(H) \subseteq E(G)$. For a set of edges $E\subseteq E(G)$, the deletion of $G$ to $E$ is the subgraph of $G$ with the same vertex set as $G$ and with the edge set $E(G)\setminus E$. If $E$ consists of a single edge $e$, we denote $G\setminus E$ simply by $G\setminus e$.

A subgraph $H$ of $G$ is called an {\em induced subgraph} if for every pair $(x,y)$ of vertices of $H$, $x$ and $y$ are adjacent in $H$ if and only if they are adjacent in $G$. Clearly for any subset $V$ of $V(G)$, there exists a unique induced subgraph of $G$ with the vertex set $V$.
For a subset of vertices $V\subseteq V(G)$, the deletion of $G$ to $V$, denoted by $G\setminus V$ is the induced subgraph of $G$ on the vertex set $V(G)\setminus V$. If $V$ consists of a single vertex $x$, then we denote $G\setminus V$ simply by $G\setminus x$. 

For a graph $G$ on the vertex set $\{x_1,\ldots,x_n\}$, let $R$ be the polynomial ring on $n$ variables, also denoted by $x_1,\ldots, x_n$. The {\it edge ideal} of $G$ is 
\[
I(G)=(x_ix_j: \text{$\{x_i,x_j\}$ is an edge of $G$}) \subseteq R.
\]
We will usually refer to the symbol $x_ix_j$ (where $i\neq j, 1\le i, j\le n$) both as an edge of $G$ and as a monomial of $R$. Instead of writing $\reg_R I(G)$ and $\lind_R I(G)$, we will usually omit the obvious base ring and write simply $\reg I(G)$ and $\lind I(G)$. We refer to \cite[Chapter 6]{Vi} for a rich source of information about the theory of edge ideals.

\subsection{Induced matchings}
The graph with $2m$ vertices $x_1,\ldots,x_m, y_1,\ldots,y_m$ (where $m\ge 1$) and exactly $m$ edges $x_1y_1,\ldots,x_my_m$ is called the $mK_2$ graph. 

Let $G$ be a graph. A {\it matching} of $G$ is a collection of edges $e_1,\ldots,e_m$ (where $m\ge 1$) such that no two of them have a common vertex. The number $m$ is then called the {\it size} of the matching. A matching $e_1,\ldots,e_m$ of $G$ is an {\it induced matching} if the induced subgraph of $G$ on the vertex set $e_1 \cup \cdots \cup e_m$ is the $mK_2$ graph. The induced matching number $\inmat(G)$ is defined as the largest size of an induced matching of $G$.

\subsection{Weakly chordal graphs}
\begin{defn}
A graph $G$ is said to be {\it weakly chordal} if every induced cycle in $G$ or in $G^c$ has length at most $4$. $G$ is said to be {\it chordal} if it does not contain any induced cycle of length greater than $3$. We say that $G$ is {\it co-chordal} if $G^c$ is a chordal graph.
\end{defn}
\begin{ex}
\label{ex_chordal}
Every chordal graph is weakly chordal: for such a graph $G$, clearly $G$ has no induced cycle of length greater than $3$. Moreover, $G^c$ has no induced cycle of length at least $5$. Indeed, if $G^c$ contains an induced cycle of length $5$, then taking the complement, we see that $G$ contains an induced cycle of length $5$. If $G^c$ contains an induced cycle of length at least $6$, then $G^c$ contains an induced $2K_2$. This implies that $G$ contains an induced cycle of length $4$. In any case, we get a contradiction.

From the definition, $G$ is weakly chordal if and only if $G^c$ is weakly chordal. Hence the above arguments also show that every co-chordal graph is weakly chordal.
\end{ex}

\begin{thm}[Fr\"oberg's theorem {\cite{Fr}}]
\label{thm_Froeberg}
Let $G$ be a graph. Then the following statements are equivalent:
\begin{enumerate}[\quad \rm(i)]
\item $\lind I(G)=0$;
\item $G$ is co-chordal;
\item $G$ is weakly chordal and $\inmat(G)=1$.
\end{enumerate}
\end{thm}
\begin{proof}
Since $I(G)$ is generated in degree $2$, (i) is equivalent to the condition that $I(G)$ has $2$-linear resolution; see Section \ref{sect_regularity}. Hence that (i) $\Longleftrightarrow$ (ii) is a reformulation of Fr\"oberg's theorem. It is not hard to see that (ii) $\Longleftrightarrow$ (iii).
\end{proof}

\section{Linearity defect and induced matchings}
\label{sect_induced_matchings}
Firstly we have the following simple change-of-rings statement.
\begin{lem}
\label{lem_basechange}
Let $(R,\mm)\to (S,\nn)$ be a morphism of noetherian local rings such that $\gr_{\nn}S$ is a flat $\gr_{\mm}R$-module. Let $M$ be a finitely generated $R$-module. Then there is an equality
\[
\lind_R M =\lind_S (M\otimes_R S).
\]
\end{lem}
\begin{proof}
Let $F$ be the minimal free resolution of $M$ over $R$. Since $\gr_{\mm}R\to \gr_{\nn}S$ is a flat morphism, so is the map $R\to S$. Hence $F\otimes_R S$ is a minimal free resolution of $M\otimes_R S$ over $S$. Observe that we have an isomorphism of complexes of $\gr_{\nn}S$-modules
\begin{equation}
\label{eq_complexiso}
 \linp^S(F\otimes_R S) \cong \linp^R F \otimes_{\gr_{\mm}R} (\gr_{\nn}S).
\end{equation}
Indeed, at the level of modules, we wish to show
\begin{equation}
\label{eq_iso}
\left(\linp^S(F\otimes_R S)\right)_i \cong \left(\linp^R F\right)_i \otimes_{\gr_{\mm}R} (\gr_{\nn}S)
\end{equation}
as graded $\gr_{\nn}S$-modules for each $i\ge 0$. On the one hand, there is the following chain in which the first isomorphism follows from \eqref{eq_linp},
\begin{align*}
(\linp^R F)_i \otimes_{\gr_{\mm}R} (\gr_{\nn}S) &\cong \left(\frac{F_i}{\mm F_i} \otimes_{R/\mm} (\gr_{\mm}R)(-i)\right) \otimes_{\gr_{\mm}R} (\gr_{\nn}S) \\
&\cong \frac{F_i}{\mm F_i} \otimes_{R/\mm} (\gr_{\nn}S)(-i)\\
&\cong \left(\frac{F_i}{\mm F_i} \otimes_{R/\mm} (S/\nn)\right)\otimes_{S/\nn} (\gr_{\nn}S)(-i).
\end{align*}
On the other hand, also from \eqref{eq_linp}, there is an isomorphism
\[
\left(\linp^S (F\otimes_R S)\right)_i \cong \frac{F_i\otimes_R S}{\nn(F_i\otimes_R S)} \otimes_{S/\nn} (\gr_{\nn}S)(-i).
\]
Since $F_i$ is a free $R$-module, we have a natural isomorphism of $(S/\nn)$-modules
\[
\frac{F_i\otimes_R S}{\nn(F_i\otimes_R S)} \cong \frac{F_i}{\mm F_i} \otimes_{R/\mm} (S/\nn).
\]
Hence the isomorphisms of type \eqref{eq_iso} were established. We leave it to the reader to check that such isomorphisms give rise to an isomorphism at the level of complexes.

Since $\gr_{\mm}R\to \gr_{\nn}S$ is a morphism of standard graded algebras, it is also faithfully flat. Hence the isomorphism \eqref{eq_complexiso} gives us
\[
\sup \{i: H_i(\linp^R F)\neq 0\} = \sup \{i: H_i(\linp^S (F\otimes_R S))\neq 0\},
\]
which is the desired conclusion.
\end{proof}
\begin{cor}
\label{cor_omit_subscripts}
Let $(R,\mm)\to (S,\nn)$ be a flat extension of standard graded $k$-algebras, and $I$ a homogeneous ideal of $R$. Then there are equalities
\begin{align*}
\lind_R I &=\lind_S (IS),\\
\reg_R I &=\reg_S (IS).
\end{align*}
\end{cor}
\begin{proof}
For the first equality, use the graded analog of Lemma \ref{lem_basechange} for the map $R \to S$ and note that $\gr_{\mm}R \cong R$. The second equality follows from the same line of thought and is even simpler.
\end{proof}
Hence below, especially in Sections \ref{sect_weaklychordal} and \ref{sect_cycles}, whenever we work with a polynomial ring $S$, a polynomial subring $R$ and an ideal $I$ of $R$, there is no danger of confusion in writing simply $\lind I$ instead of $\lind_R I$ or $\lind_S (IS)$. The same remark applies to the regularity.

Now we prove that for any graph $G$, $\lind I(G)$ is bounded below by $\inmat(G)-1$.  Although this inequality is simple, it becomes an equality for a non-trivial class of graphs -- weakly chordal graphs. This will be proved in Theorem \ref{thm_weaklychordal}.

Recall that a ring homomorphism $\theta: R\to S$ is called an {\it algebra retract} if there exists a local homomorphism $\vphi: S \to R$ such that $\vphi\circ \theta$ is the identity map of $R$. In such a case, $\vphi$ is called the {\it retraction map} of the algebra retract $\theta$. If $R, S$ are graded rings, we require $\theta, \vphi$ to preserve the gradings. The following result can be proved in the same manner as \cite[Lemma 4.7]{NgV}. Note that the proof of {\it ibid.} depends critically on a result of \c{S}ega \cite[Theorem 2.2]{Se} which was stated for local rings, but holds in the graded case as well.
\begin{lem}
\label{lem_retract}
Let $\theta: (R,\mm)\to (S,\nn)$ be an algebra retract of standard graded $k$-algebras with the retraction map $\vphi: S\to R$. Let $I\subseteq \mm$ be a homogeneous ideal of $R$. Let $J\subseteq \nn$ be a homogeneous ideal containing $\theta(I)S$ such that $\vphi(J)R=I$. Then there are inequalities
\begin{align*}
\lind_R (R/I) &\le \lind_S (S/J),\\
\lind_R I &\le \lind_S J,\\
\reg_R I &\le \reg_S J.
\end{align*}
\end{lem}
\begin{cor}
\label{cor_inducedsubgr}
Let $G$ be a graph and $H$ an induced subgraph. Then there are inequalities 
\begin{align*}
\reg I(H) &\le \reg I(G),\\
\lind I(H) &\le \lind I(G).
\end{align*}
\end{cor}
\begin{proof}
Straightforward application of Lemma \ref{lem_retract}.
\end{proof}
A direct consequence is the following statement. The bound for the regularity in the next corollary is well-known; see \cite[Lemma 2.2]{K}, \cite[Lemma 7]{W}.
\begin{cor}
\label{cor_indmatch}
Let $G$ be a graph. Then there are inequalities 
\begin{align*}
\lind I(G) &\ge \inmat(G)-1,\\
\reg I(G) &\ge \inmat(G)+1.
\end{align*}
\end{cor}
\begin{proof}
Let $m=\inmat(G)$ and $x_1y_1,\ldots,x_my_m$ be an induced matching of $G$. Then by Corollary \ref{cor_inducedsubgr},
\[
\lind I(G)\ge \lind (x_1y_1,\ldots,x_my_m)=m-1=\inmat(G)-1,
\]
where the second equality follows from direct inspection. The same proof works for the regularity.
\end{proof}
It is of interest to find good upper bounds for the linearity defect of edge ideals in terms of the combinatorics of their associated graphs. A trivial bound exists: using Taylor's resolution \cite[Section 7.1]{HH2}, we have that $\lind I(G)$ is not larger than the number of edges of $G$ minus $1$. 
\section{Morphisms which induce trivial maps of Tor}
\label{sect_Betti_split}
\subsection{Exact sequence estimates}
The following result will be employed several times in the sequel. It was proved using \c{S}ega's interpretation of the linearity defect in terms of Tor modules \cite[Theorem 2.2]{Se}.
\begin{lem}[Nguyen, {\cite[Proposition 2.5 and Corollary 2.10]{Ng}}]
\label{lem_exactseq}
Let $(R,\mm,k)$ be a noetherian local ring. Let $0\longrightarrow M \longrightarrow P \longrightarrow N\longrightarrow 0$ be an exact sequence of finitely generated $R$-modules. Denote 
\begin{align*}
d_P&=\inf\{m\ge 0: \Tor^R_i(k,M) \longrightarrow \Tor^R_i(k,P) ~\text{is the trivial map for all $i\ge m$}\},\\
d_N &=\inf\{m\ge 0: \Tor^R_i(k,P) \longrightarrow \Tor^R_i(k,N) ~\text{is the trivial map for all $i\ge m$}\},\\
d_M&=\inf\{m\ge 0: \Tor^R_{i+1}(k,N) \longrightarrow \Tor^R_i(k,M) ~\text{is the trivial map for all $i\ge m$}\}.
\end{align*}
\begin{enumerate}[\quad \rm(i)]
\item There are inequalities
\begin{align*}
\lind_R P &\le \max\{\lind_R M, \lind_R N, \min\{d_M, d_N\}\},\\
\lind_R N &\le \max\{\lind_R M+1, \lind_R P, \min\{d_M+1,d_P\}\},\\
\lind_R M &\le \max\{\lind_R P, \lind_R N-1, \min\{d_P,d_N-1\}\}.
\end{align*}
\item If moreover $P$ is a free module, then $\lind_R M=\lind_R N-1$ if $\lind_R N\ge 1$ and $\lind_R M=0$ if $\lind_R N=0$.
\end{enumerate}
\end{lem}
\begin{rem}
Unfortunately, in general without the correcting terms $d_M, d_N, d_P$, none of the ``simplified'' inequalities
\begin{align*}
\lind_R P &\le \max\{\lind_R M, \lind_R N \},\\
\lind_R N &\le \max\{\lind_R M+1, \lind_R P \},\\
\lind_R M &\le \max\{\lind_R P, \lind_R N-1\}.
\end{align*}
is true. See \cite[Example 2.9]{Ng} for details.
\end{rem}

We state now the main technical result of this section. Although the second inequality will not be employed much in the sequel, it might be of independent interest.
\begin{prop}
\label{prop_zeromap}
Let $(R,\mm,k)$ be a noetherian local ring. Let $M \xlongrightarrow{\phi} P$ be an injective morphism of finitely generated $R$-modules, and $N=\Coker \phi$. Assume that $\Tor^R_i(k,M)\xrightarrow{\Tor^R_i(k,\phi)} \Tor^R_i(k,P)$ is the trivial map for all $i\ge \max\{\lind_R M,\lind_R P-1\}$. Then there are inequalities:
\begin{align*}
\lind_R N &\le \max\{\lind_R M+1,\lind_R P\},\\
\lind_R P &\le \max\{\lind_R M,\lind_R N\}.
\end{align*}
If additionally, the map $\Tor^R_i(k,\phi)$ is trivial for all $i\ge 0$, then there is one further inequality
\[
\lind_R M\le \max\{\lind_R P, \lind_R N-1\}.
\]
\end{prop}
\begin{proof}
We have an exact sequence of $R$-modules
\[
0\longrightarrow M \longrightarrow P \longrightarrow N \longrightarrow 0.
\]
Consider the number 
$$
d_P=\inf\{m\ge 0: \Tor^R_i(k,\phi) ~\text{is the trivial map for all $i\ge m$}\}.
$$
By Lemma \ref{lem_exactseq}, there are inequalities
\begin{align*}
\lind_R N &\le \max\{\lind_R P,\lind_R M+1,d_P\},\\
\lind_R M &\le \max\{\lind_R P,\lind_R N-1,d_P\}.
\end{align*}
By the hypothesis, $d_P \le \max\{\lind_R M,\lind_R P-1\}$. Hence $\lind_R N \le \max\{\lind_R M+1,\lind_R P\}$. In the second part of the statement, $d_P=0$, hence $\lind_R M \le \max\{\lind_R P, \lind_R N-1\}$.

It remains to show that if $d_P \le \max\{\lind_R M,\lind_R P-1\}$ then $\lind_R P\le \max\{\lind_R M,\lind_R N\}.$

If $\lind_R P\le \lind_R M$, then we are done. Assume that $\lind_R P=s\ge \lind_R M+1$, then $s\ge 1$. 
	
Let $F,G$ be the minimal free resolution of $M, P$, respectively. Let $\vphi:F\longrightarrow G$ be a lifting of $\phi: M\longrightarrow P$. By the hypothesis, $\vphi\otimes_R k$ is the zero map for all $i\ge s-1$, hence $\vphi(F_i)\subseteq \mm G_i$ for all such $i$.

The mapping cone of $\vphi$ is a free resolution of $N$. Let $L$ be this mapping cone, then
$L=G\oplus F[-1]$. Note that $L$ is not necessarily minimal. Consider the complex
\[
\xymatrixcolsep{5mm}
\xymatrixrowsep{2mm}
\xymatrix{
H=L_{\ge s-1}:  \cdots \ar[rr]&&  L_i  \ar[rr] &&  L_{i-1} \ar[rr]&& \cdots \ar[rr]&& L_s \ar[rr] && L_{s-1} \ar[rr]&& 0,
}
\]
and $U=\Coker (L_s\to L_{s-1})$. Let $\shift^{s-1} U$ denote the complex with $U$ in homological position $s-1$ and $0$ elsewhere. Then $H$ is a minimal free resolution of $\shift^{s-1} U$, as $\vphi(F_i)\subseteq \mm G_i$ for all $i\ge s-1$. 

It is enough to show that $\lind_R (\shift^{s-1} U)\ge s$, namely $\lind_R U\ge 1$. Indeed, using the fact that $\lind_R U\ge 1$ and applying repeatedly Lemma \ref{lem_exactseq}(ii) for the complex
\[
\xymatrixcolsep{5mm}
\xymatrixrowsep{2mm}
\xymatrix{
	0 \ar[rr] &&  U \ar[rr]&&  L_{s-2}  \ar[rr] &&  \cdots \ar[rr]&& L_1 \ar[rr] && L_0 \ar[rr]&& N \ar[rr] && 0,
}
\]
we obtain $\lind_R N=\lind_R U+s-1\ge s$. This gives us $\lind_R P\le \lind_R N$, as desired.

  Since $H_s(\linp^R G)\neq 0$, there exists a cycle $\overline{u}\in \mm^iG_s/(\mm^{i+1}G_s)$ in $(\linp^R G)_s$ which is not a boundary of $\linp^R G$ (where $i\ge 0$). Let us show that the cycle
$$
(\overline{u},0)\in \frac{\mm^iG_s}{\mm^{i+1}G_s} \bigoplus \frac{\mm^iF_{s-1}}{\mm^{i+1}F_{s-1}},
$$ 
in $(\linp^R H)_s$ is not a boundary of $\linp^R H$. Assume the contrary is true. Denoting by $\partial^F,\partial^G, \partial^H$ the differential of $F,G,H$, respectively, then there exists 
$$
(\overline{u'},\overline{v'})\in \frac{\mm^{i-1}G_{s+1}}{\mm^iG_{s+1}} \bigoplus \frac{\mm^{i-1}F_s}{\mm^iF_s} \subseteq (\linp^R H)_{s+1}
$$ 
such that 
\[
(\overline{u},0)=\partial^H(\overline{u'},\overline{v'})=(\partial^G(\overline{u'})+\vphi(\overline{v'}),\partial^F(\overline{v'})).
\]
We have $\partial^F(\overline{v'})=0$. But $\overline{v'}\in (\linp^R F)_s$ and $s\ge \lind_R M+1$, so $v'-\partial^F(v'')\in \mm^iF_s$ for some $v''\in \mm^{i-2}F_{s+1}$. Applying $\vphi$, we obtain
\[
\vphi(v')-\partial^G(\vphi(v'')) =\vphi(v')-\vphi(\partial^F(v''))\in \mm^i\vphi(F_s).
\]
Observe that $\vphi(F_s)\subseteq \mm G_s$ by the above argument. Hence the last chain gives $\vphi(v') \equiv \partial^G(\vphi(v''))$ modulo $\mm^ {i+1}G_s$.

Combining the last statement with the equality $\overline{u}=\partial^G(\overline{u'})+\vphi(\overline{v'})$, we have the following congruence modulo $\mm^{i+1}G_s$:
\[
0 \equiv u-\partial^G(u')-\vphi(v') \equiv u-\partial^G(u'-\vphi(v'')).
\]
Recall that $v''\in \mm^{i-2}F_{s+1}$, hence $\vphi(v'')\in \mm^{i-2}\vphi(F_{s+1})\subseteq \mm^{i-1}G_{s+1}$. The last inclusion follows from the fact that $\vphi(F_i)\subseteq \mm G_i$ for $i\ge s-1$.

Hence $\overline{u}\in \mm^iG_s/(\mm^{i+1}G_s)$ equals to the image of $\overline{u'-\vphi(v'')} \in \mm^{i-1}G_{s+1}/(\mm^iG_{s+1})$. This contradicts with the condition that $\overline{u}$ is not a boundary of $\linp^R G$. Hence $(\overline{u},0)$ is not a boundary of $\linp^R H$, and $H_s(\linp^R H)\neq 0$, as desired.

The last fact yields $\lind_R (\shift^{s-1} U)\ge s$, finishing the proof.
\end{proof}

\subsection{Betti splittings}
Let $(R,\mm, k)$ be a noetherian local ring. For a finitely generated $R$-module $M$, denote $\beta_i(M)=\dim_k \Tor^R_i(k,M)$ its $i$-th Betti number. Observe that $\beta_0(M)$ equals the minimal number of generators of $M$. The following statement is an analog of \cite[Proposition 2.1]{FHV} and admits the same proof.

\begin{lem}
\label{lem_equivalence}
Let $(R,\mm,k)$ be a noetherian local ring and $I, J, K \subseteq \mm$ are ideals such that $I=J+K$. The following statements are equivalent:
\begin{enumerate}[\quad \rm(i)]
\item for all $i\ge 0$, the natural maps $\Tor^R_i(k, J\cap K) \longrightarrow \Tor^R_i(k,J)$ and $\Tor^R_i(k, J\cap K) \longrightarrow \Tor^R_i(k, K)$ are trivial;
\item for all $i\ge 0$, the equality
\[
\beta_i(I)=\beta_i(J)+\beta_i(K)+\beta_{i-1}(J\cap K)
\]
holds;
\item the mapping cone construction for the map $J\cap K \longrightarrow J\oplus K$ yields a minimal free resolution of $I$.
\end{enumerate}
\end{lem}
The following concept is particularly useful to our purpose. It is a straightforward generalization of the notion introduced by Francisco, H\`a, and Van Tuyl in \cite[Definition 1.1]{FHV}. Its modification for graded algebras and homogeneous ideals is routine.
\begin{defn}
\label{defn_Betti_splitting}
Let $(R,\mm,k)$ be a noetherian local ring and $I\subseteq \mm$ an ideal. A decomposition $I=J+K$ of $I$, where $J, K\subseteq \mm$ are ideals such that $\beta_0(I)=\beta_0(J)+\beta_0(K)$, is called a {\it Betti splitting} if one of the equivalent conditions in Lemma \ref{lem_equivalence} holds. 
\end{defn}
\begin{rem}
Strictly speaking, we do not need the condition $\beta_0(I)=\beta_0(J)+\beta_0(K)$ in Definition \ref{defn_Betti_splitting} because of Lemma \ref{lem_equivalence}: in part (ii) of that lemma, choosing $i=0$, we get $\beta_0(I)=\beta_0(J)+\beta_0(K)$. However, we include this condition to stress the special feature of the decomposition of $I$ as $J+K$.
\end{rem}

We record the following example of a Betti splitting for the sake of clarity.
\begin{ex}
Let $(R_1,\mm), (R_2,\nn)$ be standard graded $k$-algebras and $J\subseteq \mm, K\subseteq \nn$ be homogeneous ideals. Let $I=J+K \subseteq R=R_1\otimes_k R_2$. We claim that the decomposition $I=J+K$ is a Betti splitting.

First, notice the following fact: Let $M, N$ be finitely generated graded modules over $R_1,R_2$, respectively. Let $F_M, F_N$ be the  minimal free resolutions of $M, N$ over $R_1,R_2$. Then $F_M \otimes_k F_N$ is a minimal free resolution of $M\otimes_k N$ over $R$. Next, consider the short exact sequence
\[
0\longrightarrow JK \longrightarrow J \longrightarrow \frac{J}{JK} \cong J\otimes_k \frac{R_2}{K} \longrightarrow 0.
\]
Let $F$ and $G$ be the minimal free resolutions of $J$ and $R_2/K$ over $R_1$ and $R_2$, in that order. The map $F \longrightarrow F\otimes_k G$ naturally yields a lifting of the natural map $J \longrightarrow J\otimes_k (R_2/K)$. This implies that the map $\Tor^R_i(k,J) \longrightarrow \Tor^R_i(k, J/(JK))$ is injective for all $i\ge 0$. From the above exact sequence we get $\Tor^R_i(k, JK) \longrightarrow \Tor^R_i(k, J)$ is trivial for all $i\ge 0$. It is an elementary fact that $JK=J\cap K$ in $R$, thus the map $\Tor^R_i(k, J\cap K) \longrightarrow \Tor^R_i(k, J)$ is also trivial for $i\ge 0$. Similar arguments apply for the map   $\Tor^R_i(k, J\cap K) \longrightarrow \Tor^R_i(k, K)$, hence the decomposition $I=J+K$ is a Betti splitting.

It is not hard to show that if $R_1,R_2$ are polynomial rings and $J, K$ are monomial ideals over them, then the decomposition $I=J+K$ is an Eliahou-Kervaire splitting in the sense of \cite{FHV}.
\end{ex}
Given a Betti splitting $I=J+K$, the projective dimension and regularity (in the graded case) of $I$ can be read off from that of $J, K$ and $J\cap K$. 
\begin{cor}[See H\`a and Van Tuyl {\cite[Theorem 2.3]{HV1}}]
Let $(R,\mm)$ be a noetherian local ring and $I\subseteq \mm$ an ideal with a Betti splitting $I=J+K$. Then there is an equality
\[
\projdim_R I=\max\{\projdim_R J, \projdim_R K, \projdim_R (J\cap K)+1\}.
\]
If moreover $R$ is a standard graded $k$-algebra and $I, J, K$ are homogeneous ideals then we also have
\[
\reg_R I=\max\{\reg_R J, \reg_R K, \reg_R (J\cap K)-1\}.
\]
\end{cor}
\begin{proof}
Straightforward from (the graded analog of) Lemma \ref{lem_equivalence}(ii).
\end{proof}
Interestingly, the linearity defect stays under control as well in the presence of a Betti splitting.
\begin{thm}
\label{thm_Betti_splitting}
Let $(R,\mm)$ be a noetherian local ring, $I\subseteq \mm$ an ideal with a Betti splitting $I=J+K$. Then there are inequalities
\begin{align*}
\lind_R I &\le \max\{\lind_R J, \lind_R K,\lind_R (J\cap K)+1\},\\
\max\{\lind_R J,\lind_R K\} &\le \max\{\lind_R (J\cap K), \lind_R I\},\\
\lind_R (J\cap K) &\le \max\{\lind_R J, \lind_R K,\lind_R I-1\}.
\end{align*}
\end{thm}
\begin{proof}
Applying Proposition \ref{prop_zeromap} to the natural inclusion $J\cap K \longrightarrow J\oplus K$, we get the desired conclusion.
\end{proof}
We do not know any example of a Betti splitting $I=J+K$ for which the inequality $\max\{\lind_R J, \lind_R K\} \le \lind_R I$ does not hold.

A class of Betti splittings is supplied by the following
\begin{prop}
\label{prop_Koszul_decomp}
Let $(R,\mm)$ be a noetherian local ring and $I\subseteq \mm$ an ideal. Let $J,L\subseteq \mm$ be ideals and $0\neq y\in \mm$ be an element such that $I=J+yL$ and the following conditions are satisfied:
\begin{enumerate}[\quad \rm(i)]
\item $L$ is a Koszul ideal,
\item $J\cap L\subseteq \mm L$,
\item $y$ is a regular element w.r.t.~ $(R/J)$,
\item $y$ is a regular element w.r.t.~ $J$ and $L$, e.g. $R$ is a domain.
\end{enumerate}
Then the decomposition $I=J+yL$ is a Betti splitting. Moreover, there are inequalities
\begin{align*}
\lind_R I &\le \max\{\lind_R J, \lind_R (J\cap L)+1\},\\
\lind_R J &\le \max\{\lind_R (J\cap L), \lind_R I\},\\
\lind_R (J\cap L) &\le \max\{\lind_R J, \lind_R I-1\}.
\end{align*}
\end{prop}
The following lemma is useful for the proof of Proposition \ref{prop_Koszul_decomp} as well as that of Theorem \ref{thm_weaklychordal}.
\begin{lem}[Nguyen, {\cite{Ng}}]
\label{lem_persistence}
Let $(R,\mm)$ be a noetherian local ring, and $M\xlongrightarrow{\phi} P$ be a morphism of finitely generated $R$-modules. Then the following statements hold: 
\begin{enumerate}
\item[\textup{(a)}] If for some $\ell \ge \lind_R M$, the map $\Tor^R_i(k,M) \longrightarrow \Tor^R_i(k,P)$ is injective at $i=\ell$, then that map is also injective for all $i\ge \ell$.
\item[\textup{(a1)}] In particular, if $M$ is a Koszul module and $\phi^{-1}(\mm P)=\mm M$ then the natural map $\Tor^R_i(k,M) \longrightarrow \Tor^R_i(k,P)$ is injective for all $i\ge 0$.
\item[\textup{(b)}] If for some $\ell \ge \lind_R P$, the map $\Tor^R_i(k,M) \longrightarrow \Tor^R_i(k,P)$ is trivial at $i=\ell$, then that map is also trivial for all $i\ge \ell$.
\item[\textup{(b1)}] In particular, if $P$ is a Koszul module and $\phi(M) \subseteq \mm P$ then the map $\Tor^R_i(k,M) \longrightarrow \Tor^R_i(k,P)$ is trivial for all $i\ge 0$.
\end{enumerate}
\begin{proof}
For (a) and (b), see \cite[Lemma 2.8]{Ng}. 

For (a1), note that if $\phi^{-1}(\mm P)=\mm M$ then the map $\Tor^R_i(k,M) \longrightarrow \Tor^R_i(k,P)$ is injective at $i=0$. Since $M$ is Koszul, the conclusion follows from (a). Similar arguments work for (b1).
\end{proof}

\end{lem}

\begin{proof}[Proof of Proposition \ref{prop_Koszul_decomp}]
Consider the following short exact sequence, in which the equality holds because of the assumption (iii):
\[
\xymatrixcolsep{5mm}
\xymatrixrowsep{2mm}
\xymatrix{
0 \ar[rr]&&  J\cap yL=y(J\cap L)  \ar[rr] &&  J\oplus yL \ar[rr] && I \ar[rr]&& 0.
}
\]
Since $J\cap L\subseteq J$, the map $\Tor^R_i(k,y(J\cap L))\to \Tor^R_i(k,J)$ factors through $\Tor^R_i(k,yJ)\to \Tor^R_i(k,J)$. The last map is trivial since $y$ is $J$-regular. Hence $\Tor^R_i(k,y(J\cap L))\to \Tor^R_i(k,J)$ is also the trivial map.

Since $y(J\cap L)\cong J\cap L, yL\cong L$ (by the assumption (iv)), the map $\Tor^R_i(k,y(J\cap L))\to \Tor^R_i(k,yL)$ can be identified with the map $\Tor^R_i(k,J\cap L)\to \Tor^R_i(k,L)$. Since $L$ is Koszul and $J\cap L\subseteq \mm L$, the last map is trivial by Lemma \ref{lem_persistence}(b1). Therefore $\Tor^R_i(k,y(J\cap L))\to \Tor^R_i(k,yL)$ is also the trivial map. This means that $I=J+yL$ is a Betti splitting.

For the remaining inequalities, note that $J\cap yL=y(J\cap L)\cong J\cap L$, so Theorem \ref{thm_Betti_splitting} and the fact that $\lind_R L=0$ yield the desired conclusion.
\end{proof}
Let $R=k[x_1,\ldots,x_n]$ be a polynomial ring, $I$ a monomial ideal and $y$ one of the variables. The ideal $I$ can be written as $I=J+yL$, where $J$ is generated by monomials in $I$ which are not divisible by $y$ and $yL$ is the ideal generated by the remaining generators of $I$. If $y$ does not divide any minimal monomial generator of $I$, then $J=I$, $L=0$. Following \cite{FHV}, we say that the unique decomposition $I=J+yL$ is the {\it $y$-partition} of $I$.

Among other things, the following result generalizes the second statement of \cite[Corollary 2.7]{FHV}. It will be employed in Section \ref{sect_cycles}.
\begin{cor}
\label{cor_y-splitting}
Let $R=k[x_1,\ldots,x_n]$ be a polynomial ring over $k$, where $n\ge 1$. Let $y$ be one of the variables. Let $I$ be a monomial ideal and $I=J+yL$ its $y$-partition. Assume that $L$ is Koszul. Then $I=J+yL$ is a Betti splitting and there are inequalities
\begin{align*}
\lind_R I &\le \max\{\lind_R J, \lind_R (J\cap L)+1\},\\
\lind_R J &\le \lind_R I,\\
\lind_R (J\cap L) &\le \max\{\lind_R J, \lind_R I-1\}.
\end{align*}
In particular, either $\lind_R I=\lind_R J\ge \lind_R (J\cap L)$, or $\lind_R I=\lind_R (J\cap L)+1 \ge \lind_R J+1$.
\end{cor}
\begin{proof}
It is straightforward to check that the conditions of Proposition \ref{prop_Koszul_decomp} are satisfied. Hence the first and third inequalities follow. The second one is a consequence of Lemma \ref{lem_retract}. The last statement is an immediate consequence of the three inequalities.
\end{proof}

\section{Weakly chordal graphs and co-two-pairs}
\label{sect_weaklychordal}
The following notion will be important to inductive arguments with edge ideals of weakly chordal graphs.
\begin{defn}
Two vertices $x, y$ of a graph $G$ form a {\it two-pair} if they are not adjacent and every induced path connecting them is of length $2$.  If $x$ and $y$ form a two-pair of $G^c$ then we say that they are a {\it co-two-pair}.
\end{defn}
Clearly if two vertices form a co-two-pair then they are adjacent. By abuse of terminology, we also say that $xy$ is a two-pair (or co-two-pair) of $G$ if the pair $x,y$ is so. 

Recall that a subset of vertices of a graph $G$ is called a {\it clique} if every two vertices in that subset are adjacent. The existence of two-pairs is guaranteed by the following result.
\begin{lem}[{Hayward, Ho\`ang, Maffray \cite[The WT Two-Pair Theorem, Page 340]{HHM}}] 
\label{lem_twopair_existence}
If $G$ is a weakly chordal graph which is not a clique, then $G$ contains a two-pair.
\end{lem}
A useful property of co-two-pairs is the following
\begin{lem}
\label{lem_copair_remove}
If $xy$ is a co-two-pair of a graph $G$, then any induced matching of $G\setminus xy$ is also an induced matching of $G$. In particular, $\inmat(G\setminus xy) \le \inmat(G)$.
\end{lem}
\begin{proof}
It suffices to prove the first part. Assume the contrary, there exists an induced matching of $G\setminus xy$ which is not an induced matching of $G$. Then necessarily, there are two edges in this matching of the form $xu$ and $yv$. But then in $G^c$, we have an induced path of length two $xvuy$. This contradicts the assumption that $xy$ is a co-two-pair. The proof is concluded.
\end{proof}

Another simple but useful property of co-two-pairs is the following
\begin{lem}
\label{lem_weaklychordal} 
Let $G$ be a  weakly chordal graph, and $xy$ be a co-two-pair of $G$. Then $G \setminus xy$ is also weakly chordal.
\end{lem}
\begin{proof}
From the definition of weak chordality, that $G$ is weakly chordal is equivalent to $G^c$ being weakly chordal. Note that $(G\setminus xy)^c$ is the addition of the edge $xy$ to $G^c$. By {\cite[Edge addition lemma, Page 185]{SS}}, the graph $(G\setminus xy)^c$ is weakly chordal. Therefore $G\setminus xy$ is weakly chordal itself.
\end{proof}
Recall that a graph $G$ is said to be weakly chordal if both $G$ and $G^c$ have no induced cycle of length $5$ or larger. We can now prove Theorem \ref{thm_main2} from the introduction.
\begin{thm}
\label{thm_weaklychordal}
Let $G$ be a weakly chordal graph with at least one edge. Then:
\begin{enumerate}[\quad \rm(i)]
\item For any co-two-pair $e$ in $G$, the decomposition $I(G)=(e)+I(G\setminus e)$ is a Betti splitting.
\item There is an equality $\lind I(G)=\inmat(G)-1.$
\end{enumerate}
\end{thm}
We remark that the argument below yields a new proof to the implication (ii) $\Longrightarrow$ (i) of Theorem \ref{thm_Froeberg} (Fr\"oberg's theorem). The reverse implication follows easily from Corollary \ref{cor_inducedsubgr} and Lemma \ref{lem_anticycle}.
\begin{proof}[Proof of Theorem \ref{thm_weaklychordal}]
For simplicity, whenever possible we will omit the obvious superscripts of Tor modules. Thanks to Corollary \ref{cor_omit_subscripts}, we will systematically omit subscripts in writing regularity and linearity defect of ideals. Since $(e)$ has $2$-linear resolution and $(e)\cap I(G\setminus e)$ is generated in degree at least $3$, the map 
$$
\Tor_i(k, (e)\cap I(G\setminus e)) \longrightarrow \Tor_i(k,(e))
$$ 
is trivial for all $i\ge 0$. Hence to prove that $I(G)=(e)+I(G\setminus e)$ is a Betti splitting, it suffices to prove that the map $\Tor_i(k, (e)\cap I(G\setminus e)) \longrightarrow \Tor_i(k,I(G\setminus e))$ is trivial for all $i\ge 0$.

We prove by induction on $|E(G)|$ and $\inmat(G)$ the following stronger result.

\medskip
\noindent
\textsf{Claim:} Let $G$ be a weakly chordal graph with at least one edge. Then the following statements hold:
\begin{enumerate}
\item[(S1)] for any co-two-pair $e$ in $G$, the natural map 
$$
\Tor_i(k,(e)\cap I(G\setminus e)) \longrightarrow \Tor_i(k,I(G\setminus e))
$$ 
is trivial for all $i\ge 0$,
\item[(S2)] $\lind I(G)=\inmat(G)-1$, and,
\item[(S3)] $\reg I(G)=\inmat(G)+1$.
\end{enumerate}
If $G$ has only one edge then the statements are clear. Assume that $G$ has at least $2$ edges and $\inmat(G)=1$. Let $e=xy$ and consider the exact sequence
\[
\xymatrixcolsep{5mm}
\xymatrixrowsep{2mm}
\xymatrix{
0 \ar[rr]&&  (xy)(I(G\setminus e): xy)   \ar[rr] && I(G\setminus e) \oplus (xy) \ar[rr]&& I(G) \ar[rr]&& 0,
} 
\]
Let $y_1,\ldots,y_p$ be elements of the set $N(x)\cup N(y) \setminus \{x,y\}$. Since $\inmat(G)=1$, any edge of $G\setminus e$ contains at least a neighbor of either $x$ or $y$. Therefore $I(G\setminus e):xy=(y_1,\ldots,y_p)$. Clearly $p\ge 1$, so the first term in the above exact sequence has regularity 
\[
\reg (y_1,\ldots,y_p)+2=3.
\]
By Lemmas \ref{lem_copair_remove} and \ref{lem_weaklychordal}, $G\setminus e$ is weakly chordal of induced matching number $1$. Hence by the induction hypothesis, $I(G\setminus e)$ has $2$-linear resolution. From the last exact sequence, we deduce that $\reg I(G)\le 2$, proving (S2) and (S3). Since $(e) \cap I(G\setminus e)$ is generated in degree at least $3$ and $\reg I(G\setminus e)=2$, statement (S1) is true by inspecting degrees. Therefore the case $|E(G)|\ge 2$ and $\inmat(G)=1$ was established.

Now assume that $|E(G)|\ge 2$ and $\inmat(G)\ge 2$. We divide the remaining arguments into several steps. 

\medskip
\noindent
\textsf{Step 1:} We derive statements (S2) and (S3) by assuming that the statement (S1) is true. 

Since $G^c$ is not a clique, by Lemma \ref{lem_twopair_existence}, there exists a co-two-pair $xy$ in $G$. Denote $e=xy, W=I(G\setminus e)$. Consider the short exact sequence
\begin{equation}
\label{eq_ses1}
\xymatrixcolsep{5mm}
\xymatrixrowsep{2mm}
\xymatrix{
0 \ar[rr]&&  (e) \cap W  \ar[rr] && W \oplus (e) \ar[rr]&& I(G) \ar[rr]&& 0.
} 
\end{equation}
Denote by $L$ the ideal generated by the variables in $N(x)\cup N(y) \setminus \{x,y\}$. Denote by $H$ the induced subgraph of $G$ on the vertex set $G\setminus (N(x)\cup N(y))$. Then
\[
(e)\cap W =e(W:e) =e(L+I(H)) \cong \left(L+I(H)\right)(-2),
\]
We have $\inmat(H) \le \inmat(G)-1$, since we can add $xy$ to any induced matching of $H$ to obtain a larger induced matching in $G$. Since $H$ is an induced subgraph of $G$, it is weakly chordal.
Using \cite[Lemma 4.10(ii)]{NgV} and the induction hypothesis for $H$, we obtain the first and second equality, respectively, in the following display
\[
\lind ((e)\cap W)=\lind I(H) =\inmat(H)-1 \le \inmat(G)-2.
\]
Similarly, there is a chain
\[
\reg ((e)\cap W)=\reg (L+I(H))+2=\reg I(H)+2=\inmat(H)+3 \le \inmat(G)+2.
\]
The second equality in the chain holds since $I(H)$ and $L$ live in different polynomial subrings. By Lemma \ref{lem_weaklychordal}, $G\setminus e$ is again weakly chordal. Now there is a chain 
\begin{equation}
\label{eq_lind_deletion}
\lind W=\lind I(G\setminus e) =\inmat(G\setminus e) -1 \le \inmat(G)-1,
\end{equation}
in which the second equality follows from the induction hypothesis, the last inequality from Lemma \ref{lem_copair_remove}. Similarly, the induction hypothesis gives
\[
\reg W=\inmat(G\setminus e) +1 \le \inmat(G)+1.
\]
Since (S1) was assumed to be true, the decomposition $I(G)=(e)+I(G\setminus e)$ is a Betti splitting. So using Theorem \ref{thm_Betti_splitting}, we get
\[
\lind I(G) \le \max\{\lind ((e) \cap W)+1, \lind W, \lind (e)\} \le \inmat(G)-1.
\]
Together with Corollary \ref{cor_indmatch}, we get (S2). 

From the sequence \eqref{eq_ses1}, we also see that
\[
\reg I(G) \le \max\{\reg ((e) \cap W)-1, \reg W, \reg (e)\} \le \inmat(G)+1.
\]
The reverse inequality is true by Corollary \ref{cor_indmatch}, thus we obtain (S3).

\medskip
\noindent
\textsf{Step 2:} Set $g=\inmat(G)$, then $g\ge 2$ by our working assumption. In order to prove (S1), it suffices to prove the following weaker statement: the map
\[
\Tor_i(k,(e)\cap W) \longrightarrow \Tor_i(k,W)
\]
is trivial for all $i \le g-1$. 

Indeed, since $e$ is a co-two-pair, by Lemma \ref{lem_weaklychordal}, we get that $G\setminus e$ is weakly chordal. Furthermore, similarly to the chain \eqref{eq_lind_deletion} in Step 1, we have $\lind W =\inmat(G\setminus e) -1 \le g-1$. Combining the last inequality with the statement, we deduce that
\[
\Tor_i(k,(e)\cap W) \longrightarrow \Tor_i(k,W)
\]
is trivial for all $i \le \lind W$. An application of Lemma \ref{lem_persistence}(b) for the map $(e)\cap W \longrightarrow W$ implies that 
\[
\Tor_i(k,(e)\cap W) \longrightarrow \Tor_i(k,W)
\]
is also trivial for all $i\ge \lind W$.

\medskip
\noindent
\textsf{Step 3:} It remains to prove the statement that the map
\[
\phi: \Tor_i(k,(e)\cap W) \longrightarrow \Tor_i(k,W)
\]
is trivial for $i\le g-1$. Since $G\setminus e$ is weakly chordal and $\inmat(G\setminus e) \le g$, we infer from the induction hypothesis for $G\setminus e$ that
\[
\reg W=\reg I(G\setminus e) =\inmat(G\setminus e)+1\le g+1.
\]
Hence it suffices to show that the map
\[
\Tor_i(k,(e)\cap W)_j \xrightarrow{\phi_j} \Tor_i(k,W)_j
\]
is trivial for all $j\le i+g+1 \le 2g$.

Since $(e)\cap W$ and $W$ are squarefree monomial ideals, it suffices to prove the claim for squarefree multidegrees; see \cite[Theorem 2.1 and Lemma 2.2]{GPW}. So we will show that $\phi_m=0$ for each squarefree monomial $m$ of degree $\le 2g$. Furthermore, by the just cited results, it is harmless to assume that $e$ divides $m$ and $\supp(m) \subseteq V(G)$. 

Let $G^{\dagger}$ be the induced subgraph of $G$ on the vertex set $\supp m$, and $W_{\le m}$ the submodule of $W$ generated by elements whose multidegree divides $m$. By \cite[Proposition 3.10]{PV}, there is a chain
\[
\Tor^R_i(k, W)_m \cong \Tor^R_i(k, W_{\le m})_m= \Tor^R_i(k, I(G^{\dagger} \setminus e))_m 
\]
for all $i\ge 0$. For the same reason, 
\[
\Tor_i^R(k, (e)\cap W)_m \cong \Tor_i^R(k, (e)\cap I(G^{\dagger}\setminus e))_m 
\]
for all $i\ge 0$. Note that $e$ is automatically a co-two-pair of $G^\dagger$, and $|V(G^\dagger)|=|\supp(m)| \le 2g$.

Denote $S=k[x_i \mid x_i \in \supp(m)]$. If $\inmat(G^\dagger)\le g-1$ then as $G^\dagger$ is also weakly chordal, the induction hypothesis for $G^\dagger$ implies that the map 
\[
\Tor^S_i(k,(e)\cap I(G^{\dagger}\setminus e)) \longrightarrow \Tor^S_i(k, I(G^{\dagger}\setminus e))
\]
is trivial for all $i\ge 0$. Since $S\longrightarrow R$ is a faithfully flat extension, we also get that
\[
\Tor^R_i(k,(e)\cap I(G^{\dagger}\setminus e)) \longrightarrow \Tor^R_i(k, I(G^{\dagger}\setminus e))
\]
is trivial for all $i\ge 0$. In particular, the last map is trivial for $i \le g-1$, which is the desired conclusion.

If $\inmat(G^\dagger)\ge g$, then having at most $2g$ vertices, there is no possibility for $G^\dagger$ other than being a $gK_2$. After relabeling the vertices of $G^\dagger$, assume that $G^\dagger$ has vertices $x_1,\ldots,x_g,z_1,\ldots,z_g$, edges $x_1z_1,\ldots,x_gz_g$, and $e=x_1z_1$. Denote $J=(x_2z_2,\ldots,x_gz_g) \subseteq S=k[x_1,\ldots,x_g,z_1,\ldots,z_g]$. What we have to show is
\[
\Tor^S_i(k,(x_1z_1)\cap J) \longrightarrow \Tor^S_i(k, J) 
\]
is trivial for $i\le g-1$. This is easy: in fact we prove the claim for all $i\ge 0$. First note that $(x_1z_1)\cap J=(x_1z_1)J$, and $x_1z_1$ is $J$-regular. Hence the map 
\[
\Tor^S_i(k,(x_1z_1)\cap J) \longrightarrow \Tor^S_i(k, J) 
\]
is nothing but the multiplication with $x_1z_1$ of $\Tor^S_i(k, J)$. The latter is a trivial map.

This finishes the induction step and the proof of the theorem.
\end{proof}
\begin{rem}
\label{rem_splitting_edge}
Let $R=k[x_1,\ldots,x_n]$ be a standard graded polynomial ring over $k$ (where $n\ge 0$), and $I$ a monomial ideal of $R$. Let $J, K$ be monomial ideals of $R$ such that $I=J+K$. A sufficient condition for the decomposition $I=J+K$ to be a Betti splitting is that $I$ is a {\it splittable ideal} with the (Eliahou-Kervaire) splitting as $J+K$; see \cite[Definition 2.3 and Theorem 2.4]{HV2}.

Following H\`a and Van Tuyl \cite[Definition 3.1]{HV2}, we say that an edge $e$ of a graph $G$ is a {\it splitting edge} if the decomposition $I(G)=(e)+I(G\setminus e)$ makes $I(G)$ into a splittable ideal. Splitting edges (of hypergraphs) are characterized in \cite[Theorem 3.2]{HV2}. By the above discussion, any splitting edge yields a Betti splitting in the sense that if $e$ is a splitting edge of $G$, then $I(G)=(e)+I(G\setminus e)$ is a Betti splitting.

Theorem \ref{thm_weaklychordal}(i) produces a new class of Betti splittings which do not come from splitting edges. For example, let $G=C_4$ and $e$ be any of its edge. The conditions of Theorem \ref{thm_weaklychordal} are satisfied, so $I(G)=(e)+I(G\setminus e)$ is a Betti splitting. Nevertheless, it is easy to check that $e$ is not a splitting edge of $G$ since it does not satisfy the condition specified in \cite[Theorem 3.2]{HV2}.
\end{rem} 

\section{Cycles}
\label{sect_cycles}
The linearity defect of edge ideals of anticycles is known; we recall the statement here. The following lemma is contained in \cite[Theorem 5.1]{OkaYan}, which calls upon the case $d=2$ of Example 4.7 in the same paper. We give a brief argument for the sake of clarity.
\begin{lem}[See Okazaki and Yanagawa {\cite[Theorem 5.1]{OkaYan}}]
\label{lem_anticycle}
Let $G$ be the anticycle of length $n\ge 4$. Then $\lind I(G)=n-3$.
\end{lem}
\begin{proof}
It is well-known, e.g. from \cite[Theorem 5.6.1]{BH}, that $R/I(G)$ is Gorenstein of dimension $2$. Therefore the resolution of $R/I(G)$ is symmetric of length $n-2$. In particular, the last differential matrix of the minimal free resolution $F$ of $R/I(G)$ is a column of elements of degree $2$. This implies that $H_{n-2}(\linp^R F)\neq 0$. Therefore $\lind_R R/I(G)=n-2$.

Since $n-2\ge 2$, the last equality implies that $\lind_R I(G)=n-3$.
\end{proof}
The main result of this section is
\begin{thm}
\label{thm_cycle}
Let $C_n$ be the cycle of length $n$, where $n\ge 3$. Then $\lind I(C_n)=2\lfloor \frac{n-2}{3} \rfloor$.
\end{thm}
\begin{proof}
We prove by induction on $n$. For simplicity, we omit the subscript concerning the ring in the notation of linearity defect. The case $n \in \{3, 4\}$ is a straightforward application of Fr\"oberg's theorem. The case $n=5$ follows from Lemma \ref{lem_anticycle}.

Assume that the conclusion is true up to $n\ge 5$, we establish it for $n+1$. Let $P_n$, as usual, be the path with edges $x_1x_2,x_2x_3,\ldots,x_{n-1}x_n$. By Corollary \ref{cor_y-splitting} we have that the $x_{n+1}$-partition $I(C_{n+1}) = I(P_n) + x_{n+1}(x_1,x_n)$ is a Betti splitting since $(x_1,x_n)$ is Koszul. By Corollary \ref{cor_y-splitting} we also get
\begin{align}
\lind I(C_{n+1}) &\le \max\{\lind I(P_n),\lind \left(I(P_n)\cap (x_1,x_n)\right) +1\}, \label{eq_ineq2}\\
\lind \left(I(P_n)\cap (x_1,x_n)\right) &\le \max\{\lind I(P_n), \lind I(C_{n+1})-1\} \label{eq_ineq3}.
\end{align}
Denote $I=I(P_n)\cap (x_1,x_n)$. Then
\[
I=x_1(x_2,x_3x_4,x_4x_5,\ldots,x_{n-2}x_{n-1})+ x_n(x_2x_3,x_3x_4,\ldots,x_{n-3}x_{n-2},x_{n-1}).
\]
By abuse of notation, denote $I(P_{n-4})=(x_3x_4,x_4x_5,\ldots,x_{n-3}x_{n-2})$. By convention, for $n=5$, $I(P_1)=0$. 
Denote $J=(x_2,x_3x_4,x_4x_5,\ldots,x_{n-2}x_{n-1})=(x_2,x_{n-2}x_{n-1})+I(P_{n-4})$, and 
$$
L=(x_2x_3,x_3x_4,\ldots,x_{n-3}x_{n-2},x_{n-1})=(x_2x_3,x_{n-1})+I(P_{n-4}).
$$ 
With the above notation, $I=x_1J+x_nL$.

\medskip
\noindent
\textsf{Claim:} The decomposition $I=x_1J+x_nL$ is a Betti splitting.

\medskip
\noindent
Indeed, consider the exact sequence
\[
\xymatrixcolsep{5mm}
\xymatrixrowsep{2mm}
\xymatrix{
0 \ar[rr]&&  x_1J\cap x_nL=x_1x_n(J\cap L)  \ar[rr] &&  x_1J\oplus x_nL \ar[rr] && I \ar[rr]&& 0.
}
\]
Take $i\ge 0$. The map $\Tor^R_i(k,x_1x_n(J\cap L))\to \Tor^R_i(k,x_1J)$ factors through $\Tor^R_i(k,x_1x_nJ)\to \Tor^R_i(k,x_1J)$, which is the trivial map. Hence the former map is also trivial. Arguing similarly for the map $\Tor^R_i(k,x_1x_n(J\cap L))\to \Tor^R_i(k,x_nL)$, we get the claim.

Note that $J\cong (x_2)+I(P_{n-3})$ and $L\cong (x_{n-1})+I(P_{n-3})$. Hence by \cite[Lemma 4.10(ii)]{NgV},
\[
\lind J=\lind L=\lind I(P_{n-3}).
\]
Thanks to Corollary \ref{cor_forest} (which is a direct consequence of Theorem \ref{thm_weaklychordal}), we then obtain 
\[
\lind J=\lind L=\left\lfloor \frac{n-2}{3} \right\rfloor-1.
\] 
Furthermore, $J\cap L=(x_2x_3,x_3x_4,\ldots,x_{n-2}x_{n-1},x_2x_{n-1}) \cong I(C_{n-2})$. Hence by the induction hypothesis,
\[
\lind (x_1J\cap x_nL)=\lind (J\cap L) = \lind I(C_{n-2})= 2\left\lfloor \frac{n-4}{3} \right\rfloor.
\]
Since $I=x_1J+x_nL$ is a Betti splitting, Theorem \ref{thm_Betti_splitting} yields the inequalities
\begin{align}
\lind I &\le \max\{\lind J, \lind L, \lind (J\cap L)+1\}, \label{eq_ineq4}\\
\lind (J\cap L) &\le \max\{\lind J, \lind L, \lind I-1\}. \label{eq_ineq5}
\end{align}
Now we distinguish three cases according to whether $n=5$, or $n=6$, or $n\ge 7$.

\medskip
\noindent
\textsf{Case 1:} Consider the case $n=5$. Now $J=(x_2,x_3x_4)$, $L=(x_2x_3,x_4)$ and $I=x_1J+x_5L=(x_1x_2,x_4x_5,x_1x_3x_4,x_2x_3x_5)$. Since $J\cap L=(x_2x_3,x_3x_4,x_2x_4)$, we see from \eqref{eq_ineq4} that $\lind I\le 1$. Using \eqref{eq_ineq2}, we obtain $\lind I(C_6)\le 2$.

Let $U=(x_1x_6,x_1x_2,x_2x_3), V=(x_3x_4,x_4x_5,x_5x_6)$, then they are ideals with $2$-linear resolutions. Clearly $I(C_6)=U+V$. Obviously $U\cap V\subseteq \mm U$ and $U\cap V\subseteq \mm V$, hence Lemma \ref{lem_persistence}(b1) implies that the decomposition $I(C_6)=U+V$ is a Betti splitting. Using Theorem \ref{thm_Betti_splitting}, we obtain an inequality
\[
\lind (U\cap V) \le \lind I(C_6)-1.
\]
If $\lind I(C_6)<2$ then $\lind (U\cap V)=0$. On the other hand, $(U\cap V)_{\left<3\right>}=(x_1x_5x_6,x_2x_3x_4)$ does not have $3$-linear resolution. This is a contradiction. So $\lind I(C_6)=2,$ as desired.

\medskip
\noindent
\textsf{Case 2:} Consider the case $n=6$. Arguing as in the case $n=5$, we obtain $\lind I(C_7) \le 2$. 

Let the presentation of $I(C_7)$ be $F/M$, where $F=R(-2)^7$ has a basis $e_1,\ldots, e_7$ such that $e_i$ maps to $x_ix_{i+1}$ for $i=1,\ldots,6$ and $e_7$ maps to $x_7x_1$. Since $M\subseteq \mm F$, obviously $M_j=0$ for $j\le 2$. It is easy to check that the following $7$ elements belong to $k$-vector space $M_3$ and they are $k$-linearly independent:
\begin{gather*}
f_1=x_3e_1-x_1e_2, f_2=x_4e_2- x_2e_3, f_3=x_5e_3-x_3e_4, f_4=x_6e_4-x_4e_5, \\
f_5=x_7e_5-x_5e_6, f_6=x_1e_6-x_6e_7, f_7=x_2e_7-x_7e_1.
\end{gather*}
There is an exact sequence of $k$-vector spaces $ 0 \longrightarrow  M_3  \longrightarrow  F_3 \longrightarrow I(C_7)_3 \longrightarrow 0.
$ It is not hard to see that $\dim_k F_3=49$ and $\dim_k I(C_7)_3=42$, hence $\dim_k M_3=7$. In particular, $M_3$ is generated by exactly the above elements. 

Let $N=M_{\left<3\right>}$. If $\lind I(C_7)\le 1$ then $M$ must be Koszul, hence $N$ must have a $3$-linear resolution. Note that $N$ is not a free module since we can check directly that
\[
x_4x_5x_6x_7f_1+x_1x_5x_6x_7 f_2+x_1x_2x_6x_7f_3+x_1x_2x_3x_7f_4+x_1x_2x_3x_4 f_5+x_2x_3x_4x_5f_6+x_3x_4x_5x_6f_7=0.
\]
In particular, $N$ has at least one non-trivial linear syzygy. So there exist linear forms $a_1,\ldots,a_7$ in $R$, not all of which are zero, such that
\[
a_1f_1+\cdots+a_7f_7=0.
\]
Looking at the coefficients of $e_1$ and $e_2$, we get
\[
a_1x_3=a_7x_7, a_1x_1=a_2x_4.
\]
This implies that $x_7$ and $x_4$ divide $a_1$, which yields $a_1=0$. From the two equations in the last display, we deduce that $a_2=a_7=0$.

Similarly, we get $a_3=\cdots=a_6=0$, a contradiction. Hence $N$ does not have a $3$-linear resolution, and thus $\lind I(C_7)\ge 2$. Hence $\lind I(C_7)=2$, as desired. 

\medskip
\noindent
\textsf{Case 3:} Now assume that $n\ge 7$. Elementary considerations show that $\lind (J\cap L)=2\left\lfloor \frac{n-4}{3} \right \rfloor>\left\lfloor \frac{n-2}{3} \right\rfloor-1=\lind J=\lind L$. Hence from the inequalities \eqref{eq_ineq4} and \eqref{eq_ineq5}, we obtain $\lind I=\lind (J\cap L)+1=2\left\lfloor \frac{n-4}{3} \right\rfloor+1$.

From the above discussions, $\lind I=2\left\lfloor \frac{n-4}{3} \right\rfloor+1 >\left\lfloor \frac{n-2}{3} \right\rfloor= \lind I(P_n)$ for all $n\ge 7$. So inspecting \eqref{eq_ineq2} and \eqref{eq_ineq3}, we obtain
\[
\lind I(C_{n+1})=\lind I+1 =2\left\lfloor \frac{n-4}{3} \right\rfloor+2 =2\left\lfloor \frac{n-1}{3} \right\rfloor.
\]
The induction step and hence the proof is now completed. 
\end{proof}

\section{Applications}
\label{sect_applications}
\subsection{Regularity}
The proof of Theorem \ref{thm_weaklychordal} yields the following consequence.
\begin{cor}[Woodroofe, {\cite[Theorem 14]{W}}]
\label{cor_Woodroofe}
Let $G$ be a weakly chordal graph with at least one edge. Then there is an equality
\[
\reg I(G)=\inmat(G)+1.
\]
\end{cor}
\begin{rem}
Woodroofe's proof of his result depends on the Kalai-Meshulam's inequality \cite{KM}, which asserts that for a polynomial ring $R=k[x_1,\ldots,x_n]$ (where $n\ge 1$), and squarefree monomial ideals $I_1,\ldots,I_m$ (where $m\ge 2$), there is an inequality
\[
\reg (I_1+I_2+\cdots +I_m) \le \reg I_1+\reg I_2+\cdots +\reg I_m -m+1.
\]
However, the following example shows that there is no hope for a straightforward analog of the Kalai-Meshulam's inequality for linearity defect.
\begin{ex}
Take $I$ to be the edge ideal of the anticycle of length $n$ where $n\ge 5$. Let $I_1$ be the edge ideal of the induced subgraph of the anticycle on the vertex set $\{x_1,\ldots,x_{n-1}\}$, and $I_2=(x_2x_n,\ldots,x_{n-2}x_n)$. Clearly $I=I_1+I_2$.
Moreover $\lind_R I_1=0$, since $I_1$ is the edge ideal of a co-chordal graph, and $\lind_R I_2=0$ since $I_2 \cong (x_2,\ldots,x_{n-2})$. On the other hand, by Lemma \ref{lem_anticycle}, $\lind_R (I_1+I_2)=n-3$.
\end{ex}
\end{rem}

\subsection{Chordal graphs}
In view of Example \ref{ex_chordal}, an immediate corollary of Theorem \ref{thm_weaklychordal} is
\begin{cor}
\label{cor_chordal}
Let $G$ be a chordal graph with at least one edge. Then there is an equality 
$$
\lind I(G)=\inmat(G)-1.
$$
\end{cor}
We can give a simplified proof of this result, using a Betti splitting statement in \cite{HV2}. 
\begin{proof}[Alternative proof of Corollary \ref{cor_chordal}]
By Corollary \ref{cor_indmatch}, it is enough to show that 
$$
\lind I(G) \le \inmat(G)-1.
$$
We use induction on $|V(G)|$ and $|E(G)|$. The case $G$ has at most $3$ vertices is immediate. It is equally easy if $|E(G)|=1$. Assume that $|V(G)|\ge 4$ and $|E(G)|\ge 2$. 

By a classical result due to Dirac, there exists a vertex of $G$ whose neighbors form a clique (such a vertex is called a {\it simplicial vertex}). Working with a connected component of $G$ with at least one edge if necessary, we can assume that the vertex in question has at least one neighbor. Assume that $x$ is a simplicial vertex and $y$ is a vertex in $N(x)$. By \cite[Lemma 5.7(i)]{HV2} and the discussion in Remark \ref{rem_splitting_edge}, $I(G)=(xy)+I(G \setminus xy)$ is a Betti splitting.

By Theorem \ref{thm_Betti_splitting}, we obtain
\begin{equation}
\label{eq_ineq_I(G)}
\lind  I(G) \le \max\{\lind  I(G \setminus xy), \lind (I(G \setminus xy)\cap (xy))+1\}.
\end{equation}
Let $L$ be the ideal generated by the variables in $N(x) \cup N(y) \setminus \{x,y\}$. Let $H$ be the induced subgraph of $G$ on the vertex set $V(G)\setminus (N(x)\cup N(y))$. Then 
$$
I(G \setminus xy)\cap (xy) =(xy)(I(G\setminus xy):xy)=(xy)(L+I(H)).
$$
In particular, \cite[Lemma 4.10(ii)]{NgV} yields the second equality in the following display
$$
\lind (I(G \setminus xy)\cap (xy))=\lind (L+I(H))=\lind I(H).
$$
Substituting in \eqref{eq_ineq_I(G)}, it follows that
\[
\lind I(G)\le \max\{\lind I(G \setminus xy), \lind I(H)+1\}.
\]
We know that $G \setminus xy$ and $H$ are also chordal graphs; see \cite[Lemma 5.7]{HV2}.  Moreover, we have $\inmat(H)\le \inmat(G)-1$ as seen in the proof of Theorem \ref{thm_weaklychordal}. 

It is routine to check that $xy$ is a co-two-pair. Thus by Lemma \ref{lem_copair_remove}, we obtain $\inmat(G \setminus xy)\le \inmat(G)$. Now by the induction hypothesis,
\[
\lind I(G)\le \max\{\lind I(G \setminus xy), \lind I(H)+1\} \le \max\{\inmat(G \setminus xy)-1, \inmat(H)\},
\]
which is not larger than $\inmat(G)-1$. The proof is now completed.
\end{proof}

Recall that $G$ is called a {\em forest} if it contains no cycle. The connected components of a forest are trees. As a consequence of Corollary \ref{cor_chordal}, we get
\begin{cor}
\label{cor_forest}
Let $G$ be a forest with at least one edge. Then $\lind I(G)=\inmat(G)-1$.

In particular, let $P_n$ be the path of length $n-1$ \textup{(}where $n\ge 2$\textup{)}, then $\lind I(P_n)=\lfloor \frac{n-2}{3}\rfloor$.
\end{cor}
\begin{proof}
For the first part, note that any forest is chordal. Hence Corollary \ref{cor_chordal} applies.

For the second, use the simple fact that $\inmat(P_n)=\lfloor \frac{n+1}{3}\rfloor$. 
\end{proof}

\subsection{Linearity defect one}
\label{sect_ld1}
Now we prove Theorem \ref{thm_main1} from the introduction. Together with Theorem \ref{thm_Froeberg}, the next result gives the extension of Fr\"oberg's theorem advertised in the abstract.
\begin{thm}
\label{thm_ld1} 
Let $G$ be a graph. Then $\lind I(G) = 1$ if and only if $G$ is weakly chordal and $\inmat(G) = 2$.
\end{thm}
\begin{proof}
For the ``only if'' direction:  By Lemma \ref{lem_anticycle} and Theorem \ref{thm_cycle}, the linearity defect of any cycle/anticycle of length at least $5$ is greater than or equal to $2$. Hence Corollary \ref{cor_inducedsubgr} implies that $G$ has to be weakly chordal. Obviously, for example by using Taylor's resolution, we have that any homogeneous syzygy of $I(G)$ is either linear or quadratic, hence the first syzygy of $I(G)$ is generated in degree at most $4$. But $\lind I(G)\le 1$, so the first syzygy of $I(G)$ is Koszul, hence from Section \ref{sect_regularity}, its regularity is also at most $4$. As $I(G)$ is generated in degree $2$, this implies that $\reg I(G) \le 3$. But $\lind I(G)>0$, so $\reg I(G) = 3$. By Corollary \ref{cor_Woodroofe}, we deduce that $\inmat(G)=2$, as desired. 

The ``if'' direction follows from Theorem \ref{thm_weaklychordal}.
\end{proof}

\subsection{Projective dimension}
\label{sect_projdim}
First, recall that a graph $G=(V,E)$ is called a bipartite graph if 
\begin{enumerate}
\item the vertex set $V$ is a disjoint union of two subsets $V_1$ and $V_2$,
\item if two vertices are adjacent then they are not both elements of $V_i$ for any $i\in \{1,2\}$. 
\end{enumerate}
In this case, the decomposition of $V$ as $V_1\cup V_2$ is called its {\it bipartite partition}. We also denote $G$ by $(V_1,V_2,E)$ given a bipartite partition $V_1\cup V_2$ of the vertex set. 

A bipartite graph $G=(V_1,V_2,E)$ is called a {\it complete bipartite graph} if $E=\{\{x,y\}: x\in V_1,y\in V_2\}$.

In this section, we will use the notion of a strongly disjoint family of complete bipartite subgraphs, introduced by Kimura \cite{Ki2}, to compute the projective dimension of edge ideals of weakly chordal graphs. For a graph $G$, we consider all families of (non-induced) subgraphs $B_1,\ldots,B_g$ of $G$ such that
\begin{enumerate}
\item each $B_i$ is a complete bipartite graph for $1\le i\le g$,
\item the graphs $B_1,\ldots,B_g$ have pairwise disjoint vertex sets,
\item there exist an induced matching $e_1,\ldots,e_g$ of $G$ for which $e_i \in E(B_i)$ for $1\le i\le g$.
\end{enumerate}
Such a family is termed a {\it strongly disjoint family of complete bipartite subgraphs}. We define
\[
d(G)=\max \left\{\sum_{i=1}^g |V(B_i)|-g \right\}
\]
where the maximum is taken over all the strongly disjoint families of complete bipartite subgraphs $B_1,\ldots,B_g$ of $G$.

The following result is a generalization of \cite[Theorem 4.1(1)]{Ki1} and \cite[Corollary 3.3]{DE}; the latter was reproved in \cite[Corollary 5.3]{Ki2}.
\begin{thm}
\label{thm_projdim_weaklychordal}
Let $G$ be a weakly chordal graph with at least one edge. Then there is an equality
\[
\projdim I(G)=d(G)-1.
\]
\end{thm}
The inequality $\projdim I(G) \ge d(G)-1$ was established by Kimura's work \cite{Ki2}. The following lemma, which might be of independent interest, is the crux in proving the reverse inequality. We are grateful to an anonymous referee for suggesting the main idea of the proof of the lemma.

\begin{lem}
\label{lem_completebipartite}
Let $x_1x_2$ be a co-two-pair of a graph $G$. Then there is a complete bipartite subgraph of $G$ with the vertex set $N(x_1)\cup N(x_2)$ \textup{(}the last union need not be the bipartite partition for that subgraph\textup{)}.
\end{lem}
\begin{proof}
Let $V$ be the set $N(x_1)\cup N(x_2)$. Define the subsets $V_{1,n}$ and $V_{2,n}$ of $V$ inductively on $n\ge 0$ as follows: $V_{1,0}=\{x_1\}, V_{2,0}=\{x_2\}$. For $n\ge 0$, we let
\[
V_{1,n+1}=V_{1,n} \cup \{z\in V:  V_{1,n} \not\subseteq N(z)\},
\]
and similarly
\[
V_{2,n+1}=V_{2,n} \cup \{z\in V: V_{2,n} \not\subseteq N(z)\}.
\]
Clearly $V_{1,n}\subseteq V_{1,n+1}$ and $V_{2,n}\subseteq V_{2,n+1}$ for all $n\ge 0$. We set $V_{1,-1}=V_{2,-1}=\emptyset$ for systematic reason. Our aim is to prove the following statements:
\begin{enumerate}
 \item $V_{1,n} \subseteq \{z\in V: V_{2,n-1} \subseteq N(z)\}$, and $V_{2,n} \subseteq \{z\in V: V_{1,n-1} \subseteq N(z)\}$,
 \item $V_{1,n}\cap V_{2,n}=\emptyset$,
 \item $G$ has a complete bipartite subgraph with the bipartite partition $V_{1,n}\cup V_{2,n}$.
\end{enumerate}
Let us use induction on $n$. If $n=0$, then (i) holds vacuously, while $V_{1,0}=\{x_1\}, V_{2,0}=\{x_2\}$ therefore (ii) and (iii) are also true. 

Assume that the statements (i) -- (iii) are true for $n\ge 0$. We establish them for $n+1$. 

\medskip
\noindent
\textsf{For (i):} if (i) was not true, we can assume that $V_{1,n+1} \not\subseteq \{z\in V: V_{2,n} \subseteq N(z)\}$. Choose $z\in V_{1,n+1}$ such that $V_{2,n} \not\subseteq N(z)$. Clearly $z\notin V_{1,n}$ because of the induction hypothesis for (iii). Hence the definition of $V_{1,n+1}$ forces $V_{1,n} \not\subseteq N(z)$. Again the last non-containment implies that $z\notin V_{2,n}$.

As $N(z)$ contains neither $V_{1,n}$ nor $V_{2,n}$, we can choose $x_{i,n}\in V_{i,n}$ such that $x_{i,n}\notin N(z)$ for $i=1,2$.

By the definition of $V_{1,n}$ we can choose $n_r\ge 0$ such that $x_{1,n}\in V_{1,n_r} \setminus V_{1,n_r-1}$ (recall that $V_{1,-1}=\emptyset$). Set $x_{1,n_r}=x_{1,n}$. Since $x_{1,n_r}\in V_{1,n_r} \setminus V_{1,n_r-1}$, there exists $x_{1,n_{r-1}} \in V_{1,n_r-1}$ such that $x_{1,n_r}$ and $x_{1,n_{r-1}}$ are not adjacent.

Continuing this argument, finally we find a sequence of indices $0=n_0<n_1<\cdots<n_r \le n$ and vertices $x_1=x_{1,n_0}, x_{1,n_1},\ldots,x_{1,n_r}=x_{1,n}$ such that $x_{1,n_i}\in V_{1,n_i}\setminus V_{1,n_i-1}$ for all $0\le i\le r$ and $x_{1,n_i}$ and $x_{1,n_{i+1}}$ are not adjacent for $0\le i\le r-1$.

Similarly, there exist a sequence of indices $0=m_0<m_1<\cdots <m_s\le n$ and vertices $x_2=x_{2,m_0},x_{2,m_1},\ldots,x_{2,m_s}=x_{2,n}$ such that $x_{2,m_j}\in V_{2,m_j}\setminus V_{2,m_j-1}$ for all $0\le j\le s$ and $x_{2,m_j}$ and $x_{2,m_{j+1}}$ are not adjacent for $0\le j\le s-1$.

Note that we have a  path, called $P$, with vertices
$$
x_1=x_{1,n_0},x_{1,n_1},\ldots,x_{1,n_r}=x_{1,n},z,x_{2,n}=x_{2,m_s},x_{2,m_{s-1}},\ldots,x_{2,m_0}=x_2
$$ 
connecting $x_1$ and $x_2$ in $G^c$ with length $r+s+2$ (see Figure \ref{figure_path}). We claim that this is an induced path and its length is $>2$.

\begin{figure}
\begin{tikzpicture}
\tikzstyle{solid node}=[circle,draw,inner sep=1.5,fill=black]
\node(0)[solid node]{}
[grow=north]
child[grow=30]{node[solid node]{}
child[grow=60]{node[solid node]{}
child[grow=90]{node[]{$\vdots$}
child{node[solid node]{}
child[grow=120]{node[solid node]{}}
}
}
}
}
child[grow=150]{node[solid node]{}
child[grow=120]{node[solid node]{}
child[grow=90]{node{$\vdots$}
child{node[solid node]{}
child[grow=60]{node[solid node]{}}
}
}
}
};
\node[below]at(0){$z$};
\node[left]at(0-2){$x_{1,n_r}$};
\node[left]at(0-2-1){$x_{1,n_{r-1}}$};
\node[left]at(0-2-1-1-1){$x_{1,n_1}$};
\node[left]at(0-2-1-1-1-1){$x_1=x_{1,n_0}$};
\node[right]at(0-1){$x_{2,m_s}$};
\node[right]at(0-1-1){$x_{2,m_{s-1}}$};
\node[right]at(0-1-1-1-1){$x_{2,m_1}$};
\node[right]at(0-1-1-1-1-1){$x_{2,m_0}=x_2$};
\end{tikzpicture}
\caption{The path $P$ in $G^c$}
\label{figure_path}
\end{figure}
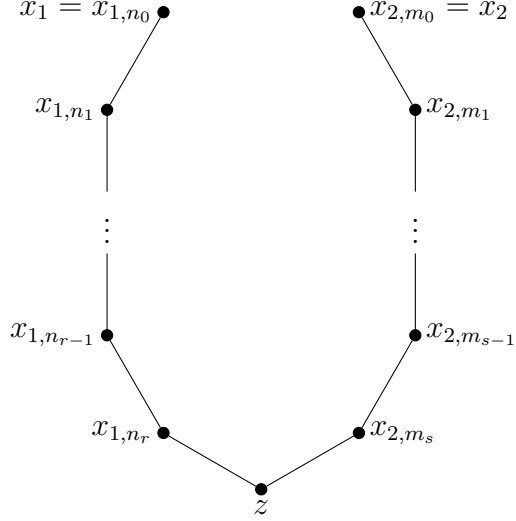

For the second part of the last claim, it suffices to observe that $r$ and $s$ cannot be both zero, otherwise $z\in V$ but $z\notin N(x_1)\cup N(x_2)$, which is absurd. For the first part, note that as $z\notin V_{1,n}\cup V_{2,n}$, we have $V_{1,n-1}\cup V_{2,n-1} \subseteq N(z)$. Hence $z$ is adjacent (in $G$) to all the vertices in $P$ except $x_{1,n}$ and $x_{2,n}$. Since $G$ has a complete bipartite subgraph with vertex set $V_{1,n}\cup V_{2,n}$, it is clear that $x_{1,n_i}$ is adjacent to $x_{2,m_j}$ for all $0\le i\le r, 0\le j\le s$. 

Now consider $0\le i, j\le r$ such that $i\le j-2$. We wish to show that $x_{1,n_i}$ and $x_{1,n_j}$ are adjacent. As $x_{1,n_j} \in V_{1,n_j}\setminus V_{1,n_j-1}$, we see that $V_{1,n_j-2} \subseteq N(x_{1,n_j})$. Since $i\le j-2$, clearly $n_i \le n_j-2$, hence $x_{1,n_i} \in V_{1,n_j-2} \subseteq N(x_{1,n_j})$. Therefore $x_{1,n_i}$ and $x_{1,n_j}$ are adjacent. Similarly, for $0\le i,j\le s$ with $i\le j-2$, the vertices $x_{2,m_i}$ and $x_{2,m_j}$ are adjacent. 

This shows that $P$ is an induced path connecting $x_1$ and $x_2$ in $G^c$ with length $>2$. But then we get a contradiction, since $x_1x_2$ is a co-two-pair.

In other words, we have $V_{1,n+1} \subseteq \{z\in V: V_{2,n}\subseteq N(z)\}$ and similarly $V_{2,n+1} \subseteq \{z\in V: V_{1,n}\subseteq N(z)\}$. This finishes the induction step for (i).

\medskip
\noindent
\textsf{For (ii):} assume that there exists $z\in V_{1,n+1}\cap V_{2,n+1}$. As we have seen,
\[
V_{2,n+1}\subseteq \{z\in V: V_{1,n}\subseteq N(z)\}.
\]
So the definition of $V_{1,n+1}$ yields $z\in V_{1,n}$. Similarly, $z\in V_{2,n}$, but then $V_{1,n}\cap V_{2,n}\neq \emptyset$, a contradiction. This finishes the induction step for (ii).

\medskip
\noindent
\textsf{For (iii):} taking $z_1\in V_{1,n+1}$ and $z_2\in V_{2,n+1}$, we want to show that $\{z_1,z_2\}\in E(G)$. First, assume that $z_1\in V_{1,n}$. From (i), we have the second inclusion in the following chain
\[
z_2\in V_{2,n+1}\subseteq \{z\in V: V_{1,n}\subseteq N(z)\},
\]
hence $z_1$ is adjacent to $z_2$. Hence it suffices to consider the case $z_1\notin V_{1,n}$ and for the same reason, we restrict ourselves to the case $z_2\notin V_{2,n}$. 

Note that $z_1\notin V_{2,n}$ since $V_{1,n+1}\cap V_{2,n}\subseteq V_{1,n+1}\cap V_{2,n+1}=\emptyset$. Hence $z_1\notin V_{1,n}\cup V_{2,n}$, and the same thing happens for $z_2$. 

Assume that on the contrary, $\{z_1,z_2\} \notin E(G)$. As in the induction step for (i), we can choose a sequence of indices $0=n_0<\cdots<n_r \le n+1$ (where $r\ge 0$) and elements $x_1=x_{1,n_0},x_{1,n_1},\ldots,x_{1,n_r}=z_1$ such that $x_{1,n_i}\in V_{1,n_i}\setminus V_{1,n_i-1}$ for $0\le i\le r$ and $x_{1,n_i}$ and $x_{1,n_{i+1}}$ are not adjacent for $0\le i\le r-1$. Similarly, we can choose a sequence of indices $0=m_0<\cdots<m_s \le n+1$ (where $s\ge 0$) and elements $x_2=x_{2,m_0},x_{2,m_1},\ldots,x_{2,m_s}=z_2$ with the similar properties. Since $z_1\notin V_{1,n}$ and $x_1\in V_{1,n}$, we have that $r\ge 1$. Analogously, $s\ge 1$.

Since $\{z_1,z_2\} \notin E(G)$, we have a path 
$$
x_1=x_{1,n_0},x_{1,n_1},\ldots,x_{1,n_r}=z_1,z_2=x_{2,m_s},x_{2,m_{s-1}},\ldots,x_{2,m_0}=x_2
$$ 
of length $r+s+1\ge 3$ connecting $x_1$ and $x_2$ in $G^c$. As in the induction step for (i), we see that the last path is an induced path in $G^c$. Again this yields a contradiction since $x_1x_2$ is a co-two-pair. So $\{z_1,z_2\} \in E(G)$, as desired. This finishes the induction step for (iii), so that all the statements (i)--(iii) are true.

Since $\{V_{1,n}\}_{n\ge 0}$ and $\{V_{2,n}\}_{n\ge 0}$ are two monotonic sequences (with respect to inclusion) consisting of subsets of the finite set $V$, there exist $q\ge 0$ such that $V_{1,n}=V_{1,n+1}, V_{2,n}=V_{2,n+1}$ for all $n\ge q$. 

If $V=V_{1,q}\cup V_{2,q}$ then we are done by statement (iii). Otherwise, consider an element $z\in V \setminus (V_{1,q}\cup V_{2,q})$. We have $z\notin V_{1,q+1}=V_{1,q}$, so $V_{1,q} \subseteq N(z)$. Enlarging $V_{2,q}$ with all such elements $z$ of $V$, again we see that there is a complete bipartite subgraph of $G$ with the vertex set $V$. Hence the lemma is now proved.
\end{proof}
\begin{proof}[Proof of Theorem \ref{thm_projdim_weaklychordal}]
That $\projdim I(G) \ge d(G)-1$ follows from \cite[Theorem 1.1]{Ki2}. We prove the reverse inequality by induction on $|E(G)|$. If $|E(G)|=1$ then there is nothing to do. Assume that $|E(G)|\ge 2.$

Let $e=x_1x_2$ be a co-two-pair of $G$, which exists because of Lemma \ref{lem_twopair_existence}. Denote $V=N(x_1)\cup N(x_2)$. Let $L$ be the ideal generated by the variables in $V \setminus \{x_1,x_2\}$ and $H$ be the induced subgraph of $G$ on the vertex set $G\setminus V$. Again we have
\[
I(G\setminus e):x_1x_2=L+I(H).
\]
Denote $p=|V|-2$; note that $p\ge 0$ since $x_1,x_2\in V$. If $I(H)=0$ then $L\neq 0$ since $|E(G)|\ge 2$. From the exact sequence
\begin{equation}
\label{eq_ses2}
\xymatrixcolsep{5mm}
\xymatrixrowsep{2mm}
\xymatrix{
0 \ar[rr]&&  (x_1x_2)(L+I(H))   \ar[rr] && I(G\setminus e) \oplus (x_1x_2) \ar[rr]&& I(G) \ar[rr]&& 0
} 
\end{equation}
we get that
\[
\projdim I(G) \le \max\{\projdim I(G\setminus e), \projdim L+1\}=\max\{\projdim I(G\setminus e), p\}.
\]
Since $G\setminus e$ is again a weakly chordal graph, by the induction hypothesis $\projdim I(G\setminus e) \le d(G\setminus e)-1\le d(G)-1$. The last inequality follows from the fact that $d(G\setminus e) \le d(G)$, which in turn follows easily from Lemma \ref{lem_copair_remove}.

By Lemma \ref{lem_completebipartite}, there is an complete bipartite subgraph $B_1$ of $G$ with the vertex set $V$.
Now $B_1$ has $p+2$ vertices, hence by the definition of $d(G)$, it follows that $(p+2)-1=p+1 \le d(G)$. Finally 
\[
\projdim I(G) \le \max\{\projdim I(G\setminus e), p\} \le d(G)-1,
\]
as desired.

Assume that $I(H)\neq 0$. Since $L$ and $I(H)$ live in different polynomial subrings of $R$ and $L$ has codimension $p$, we obtain
\[
\projdim(L+I(H)) =\projdim I(H)+p.
\]
Since $H$ is weakly chordal with fewer edges than $G$, by the induction hypothesis 
\[
\projdim I(H)=d(H)-1.
\]
Let $B_2,\ldots,B_g$ be a strongly disjoint family of complete bipartite subgraphs of $H$ which realizes $d(H)$. Note that if $e_2,\ldots,e_g$ form an induced matching of $H$, where $e_i\in B_i$, then $e,e_2,\ldots,e_g$ is an induced matching in $G$. Therefore $B_1,B_2,\ldots,B_g$ is a strongly disjoint family of complete bipartite subgraphs of $G$. In particular,
\[
d(G) \ge |V(B_1)|+\sum_{i=2}^g|V(B_i)|-g=p+2+d(H)-1=\projdim I(H)+p+2.
\]
All in all, we see that
\[
\projdim(L+I(H)) =\projdim I(H)+p \le d(G)-2.
\]
As above $\projdim I(G\setminus e)=d(G\setminus e)-1\le d(G)-1$. So from the exact sequence \eqref{eq_ses2}, we obtain
\[
\projdim I(G) \le \max\{\projdim I(G\setminus e), \projdim (L+I(H))+1\} \le d(G)-1.
\] 
This finishes the induction and the proof of the theorem.
\end{proof}
\begin{rem}
See also, e.g., \cite{KhM}, \cite{Ki1}, \cite{Ki2}, for more results about the relationship between the projective dimension of $I(G)$ and invariants coming from families of complete bipartite subgraphs of $G$.
\end{rem}

\section{Characteristic dependence}
\label{sect_dependence_characteristic}
It is well-known that the regularity of edge ideals depend on the characteristic of the field. In fact Katzman \cite{K} shows that if $R=k[x_1,\ldots,x_n]$ is a polynomial ring of dimension $n\le 10$, then any edge ideal on the vertex set $\{x_1,\ldots,x_n\}$ has characteristic-independent regularity. On the other hand, from dimension $11$ onward, there are examples of edge ideals with characteristic-dependent regularity. 

The following example is taken from Katzman's paper \cite[Page 450]{K}. It comes from a triangulation of the projective plane $\mathbb P^2_k$. The Macaulay2 package \cite{NgV2} is employed in our various computations of the linearity defect.
\begin{ex}
\label{ex_dependence}
Let $I\subseteq k[x_1,\ldots,x_{11}]$ be the following edge ideal:
\begin{gather*}
I=(x_1x_2, x_1x_6, x_1x_7, x_1x_9, x_2x_6, x_2x_8, x_2x_{10}, x_3x_4, x_3x_5, x_3x_7, x_3x_{10},\\
x_4x_5, x_4x_6, x_4x_{11}, x_5x_8, x_5x_9, x_6x_{11}, x_7x_9, x_7x_{10}, x_8x_9, x_8x_{10}, x_8x_{11}, x_{10}x_{11}).
\end{gather*}
Computations with Macaulay2 \cite{GS} show that $\lind I=3$ if $\chara k =0$ and $\lind I=7$ if $\chara k=2$.
\end{ex}
So for any $m\ge 0$, applying \cite[Lemma 4.10]{NgV}, we see that the edge ideal $I+(y_1z_1,\ldots,y_mz_m) \subseteq k[x_1,\ldots,x_{11},y_1,\ldots,y_m,z_1,\ldots,z_m]$ has linearity defect $\lind I +\lind (y_1z_1,\ldots,y_mz_m)+1=m+3$ if $\chara k=0$ and $m+7$ if $\chara k=2$. Moreover, we can replace $I+(y_1z_1,\ldots,y_mz_m)$ by the edge ideal of a connected graph using Lemma \ref{lem_cone}. Indeed, take a new vertex $y$, then the edge ideal $I+(y_1z_1,\ldots,y_mz_m)+y(x_1,\ldots,x_{11},y_1,\ldots,y_m) \subseteq k[x_1,\ldots,x_{11},y_1,\ldots,y_m,z_1,\ldots,z_m,y]$ comes from a connected graph and has the same linearity defect as $I+(y_1z_1,\ldots,y_mz_m)$.
\begin{lem}
\label{lem_cone}
Let $J\subseteq R=k[x_1,\ldots,x_n]$ be a monomial ideal which does not contain a linear form. Let $L$ be an ideal generated by variables such that $J\subseteq L$.  Consider the ideal $I=J+yL$ in the polynomial extension $S=R[y]$. There is an equality $\lind_R J=\lind_S I$.
\end{lem}
\begin{proof}
By Lemma \ref{lem_retract}, we get $\lind_R J\le \lind_S I$. For the reverse inequality, let $\pp=(x_1,\ldots,x_n,y)\subseteq S$. Consider the exact sequence
\[
0 \longrightarrow yL  \longrightarrow  I \longrightarrow \frac{J}{yL \cap J}=\frac{J}{yJ} \longrightarrow 0.
\]
The equality follows from the fact that $yL \cap J=y(L \cap J)=yJ$. Note that $yL$ is Koszul and $yL \cap \pp I=\pp yL$ by degree reasons. Hence by either \cite[Theorem 3.1]{Ng}, or Lemma \ref{lem_persistence}(a1) together with Lemma \ref{lem_exactseq}(i), we get the first inequality in the following chain
\[
\lind_S I \le \max\left\{0,\lind_S  \frac{J}{yJ}\right\}=\lind_S  \frac{J}{yJ}=\lind_S \left(J\otimes_k \frac{k[y]}{yk[y]}\right) =\lind_R J.
\]
The last equality follows from \cite[Lemma 4.9]{NgV}. This completes the proof.
\end{proof}

The linearity defect of edge ideals of bipartite graphs also may depend on the characteristic.
\begin{ex}
\label{ex_bipartite}
Dalili and Kummini \cite[Example 4.8]{DK} found an example of a bipartite graph such that the regularity of the corresponding edge ideal depends on the characteristic. Specifically, their ideal is
\begin{gather*}
I=(x_1y_1, x_2y_1, x_3y_1, x_7y_1, x_9y_1, x_1y_2, x_2y_2, x_4y_2, x_6y_2, x_{10}y_2,x_1y_3, x_3y_3, x_5y_3,\\
x_6y_3, x_8y_3, x_2y_4, x_4y_4, x_5y_4, x_7y_4, x_8y_4,x_3y_5, x_4y_5, x_5y_5, x_9y_5, x_{10}y_5,\\
x_6y_6, x_7y_6, x_8y_6, x_9y_6, x_{10}y_6) \subseteq k[x_1,\ldots,x_{10},y_1,\ldots,y_6].
\end{gather*}
Computations with Macaulay2 using our package \cite{NgV2} show that $\lind I=6$ if $\chara k=0$ and $\lind I=11$ if $\chara k=2$. 
\end{ex}

We have seen from Theorem \ref{thm_ld1} that the condition $\lind I(G)=1$ is equivalent to $G$ being weakly chordal and having induced matching number $2$. Therefore we would like to ask the following 
\begin{quest}
\label{quest_ld2}
Can the condition $\lind I(G)=2$ be characterized solely in terms of the combinatorial properties of the graph $G$, independent of the characteristic of $k$?
\end{quest}
\begin{rem}
The analog of Question \ref{quest_ld2} for regularity has a clear answer. We know from Fr\"oberg's theorem that the condition $\reg I(G)=2$ is independent of characteristic. Fern\'anderz-Ramos and Gimenez \cite[Theorem 4.1]{FG} show that if $G$ is bipartite and connected, then the following are equivalent:
\begin{enumerate}
\item $\reg I(G)=3$;
\item $G^c$ has an induced $C_4$ and the bipartite complement of $G$ has no induced $C_m$ with $m\ge 5$.
\end{enumerate}
It is not hard to see that (ii) is equivalent to the condition that the bipartite complement of $G$ is weakly chordal and $\inmat(G)=2$. Also, the reader may check that if $G$ is bipartite and disconnected, then $\reg I(G)=3$ if and only if $G$ has two connected components $G_1,G_2$, each of which is co-chordal. In particular, for a bipartite graph $G$, the condition $\reg I(G)=3$ is independent of the characteristic.

On the other hand, if $G$ is not bipartite, then the condition $\reg I(G)=3$ might depend on the characteristic: for Katzman's ideal in Example \ref{ex_dependence}, $\reg I(G)=3$ if $\chara k=0$ and $4$ if $\chara k= 2$.

Dalili and Kummini's ideal in Example \ref{ex_bipartite} shows that for connected bipartite graphs, the condition $\reg I(G)=4$ is dependent on the characteristic: for their ideal, $\reg I=4$ if $\chara k=0$ and $5$ if $\chara k=2$.
\end{rem}

In view of Example \ref{ex_bipartite}, we wonder if for a bipartite graph $G$ and $2\le \ell\le 5$, the condition $\lind I(G) =\ell$ is independent of the value of $\chara k$.

\section*{Acknowledgments}
Parts of this work were finished when the first author was visiting the Department of Mathematics, University of Nebraska -- Lincoln. We are grateful to Luchezar Avramov for his generosity and constant encouragement. The first author is grateful to Tim R\"omer for his useful comments and suggestions. Finally, we are indebted to two anonymous referees for their careful reading of the paper and for pointing a gap in the proof of Theorem \ref{thm_projdim_weaklychordal} in a previous version. We were able to fill the gap and improve the readability of our paper thanks to their thoughtful suggestions and comments.


\begin{thebibliography}{99}
\bibitem{AR}
R. Ahangari Maleki and M.E. Rossi,
\emph{Regularity and linearity defect of modules over local rings}.
J. Commut. Algebra {\bf 6}, no. 4 (2014), 485--504.

\bibitem{Avr}
L.L. Avramov,
\emph{Infinite free resolution}. 
in \emph{Six lectures on Commutative Algebra} (Bellaterra, 1996), 1--118, Progr. Math., {\bf 166}, Birkh\"auser (1998).

\bibitem{BH}
W. Bruns and J. Herzog,
\emph{Cohen-Macaulay rings.} Revised edition.
Cambridge Studies in Advanced Mathematics {\bf 39}, Cambridge University Press (1998).

\bibitem{DK}
K. Dalili and M. Kummini,
\emph{Dependence of Betti numbers on characteristic}.
Comm. Algebra {\bf 42}, no. 2 (2014), 563--570.

\bibitem{DE}
A. Dochtermann and A. Engstr\"om, 
\emph{Algebraic properties of edge ideals via combinatorial topology}. 
Electron. J. Combin. {\bf 16} (2009), Special volume in honor of Anders Bj\"orner, Research Paper 2.

\bibitem{EFS}
D. Eisenbud, G. Fl\o ystad and F.~-O. Schreyer,
\emph{Sheaf cohomology and free resolutions over exterior algebras}.
Trans. Amer. Math. Soc. {\bf 355} (2003), 4397--4426. 

\bibitem{EGHP}
D. Eisenbud, M. Green, K. Hulek and S. Popescu, 
\emph{Restricting linear syzygies: algebra and geometry}.
Compos. Math. {\bf 141} (2005), no. 6, 1460--1478.

\bibitem{FG}
O. Fern\'andez-Ramos and P. Gimenez,
\emph{Regularity $3$ in edge ideals associated to bipartite graphs}.
J. Algebr. Combin. {\bf 39} (2014), 919--937.

\bibitem{FHV}
C. Francisco, H.T. H\`a and A. Van Tuyl,
\emph{Splittings of monomial ideals}.
Proc. Amer. Math. Soc. {\bf 137} (2009), 3271--3282.

\bibitem{Fr}
R. Fr\"oberg,
\emph{On Stanley-Reisner rings}. 
In: {\it Topics in Algebra}, vol. {\bf 26} Part 2, pp. 57--70. Banach Center Publications, PWN-Polish Scientific Publishers, Warsaw (1990).


\bibitem{GPW}
V. Gasharov, I. Peeva and V. Welker,
\emph{The lcm-lattice in monomial resolutions}. 
Math. Research Letters {\bf 6} (1999), 521--532.

\bibitem{GS}
D. Grayson and M. Stillman,
\emph{Macaulay2, a software system for research in algebraic geometry.}
Available at \newblock \verb|http://www.math.uiuc.edu/Macaulay2|.

\bibitem{Ha}
H.T. H\`a,
\emph{Regularity of squarefree monomial ideals}.
In: \emph{Connections between algebra, combinatorics, and geometry},
Springer Proc. Math. Stat., {\bf 76}, Springer, New York (2014), pp. 251--276. 

\bibitem{HV1}
H.T. H\`a and A. Van Tuyl,
\emph{Splittable ideals and the resolutions of monomial ideals}.
J. Algebra {\bf 309} (2007), 405--425.

\bibitem{HV2}
H.T. H\`a and A. Van Tuyl,
\emph{Monomial ideals, edge ideals of hypergraphs, and their
graded Betti numbers}.
J. Algebr. Combin. {\bf 27} (2008), 215--245.


\bibitem{HHM}
R. Hayward, C.T. Ho\`ang and F. Maffaray,
{\em Optimizing weakly triangulated graphs}.
Graphs and Combinatorics 5 (1989), no. 1, 339--349.

\bibitem{HH}
J. Herzog and T. Hibi, 
\emph{Componentwise linear ideals}. 
Nagoya Math. J. {\bf 153} (1999), 141--153.


\bibitem{HH2}
J. Herzog and T. Hibi,
{\em Monomial ideals}.
Graduate Texts in Mathematics {\bf 260}, Springer (2011).

\bibitem{HIy}
J. Herzog and S.B. Iyengar,
\emph{Koszul modules}. J. Pure Appl. Algebra {\bf 201} (2005), 154--188.

\bibitem{HSV}
J. Herzog, L. Sharifan and M. Varbaro,
\emph{The possible extremal Betti numbers of a homogeneous ideal}. 
Proc. Amer. Math. Soc. {\bf 142} (2014), no. 6, 1875--1891. 

\bibitem{IyR}
S.B. Iyengar and T. R\"omer,
{\em Linearity defects of modules over commutative rings}.
J. Algebra {\bf 322} (2009), 3212--3237.

\bibitem{J}
C. Jacobsson, 
\emph{Finitely presented graded Lie algebras and homomorphisms of local rings}. 
J. Pure Appl. Algebra {\bf 38} (1985) 243--253.

\bibitem{KM}
G. Kalai and R. Meshulam, 
\emph{Intersections of Leray complexes and regularity of monomial ideals}. 
J. Combin. Theor. {\bf 113} (2006), 1586--1592.

\bibitem{K}
M. Katzman,
\emph{Characteristic-independence of Betti numbers of graph ideals}.
J. Combin. Theory, Ser. A {\bf 113} (2006) 435 -- 454.

\bibitem{KhM}
F. Khosh-Ahang and S. Moradi,
\emph{Regularity and projective dimension of the edge ideal of $ C_5$-free vertex decomposable graphs}.
Proc. Amer. Math. Soc. {\bf 142} (2014), 1567--1576.

\bibitem{Ki1}
K. Kimura,
\emph{Non-vanishingness of Betti numbers of edge ideals}.
In: \emph{Harmony of Gr\"obner bases and the modern industrial society}, World Sci. Publ., Hackensack, NJ (2012), pp. 
153--168.

\bibitem{Ki2}
K. Kimura,
\emph{Non-vanishingness of Betti numbers of edge ideals and complete bipartite subgraphs}.
Preprint (2013), Available online at \newblock \verb|http://arxiv.org/abs/1306.1333|. 

\bibitem{MV}
S. Morey and R. Villarreal,
\emph{Edge ideals: algebraic and combinatorial properties}.
in \emph{Progress in Commutative Algebra, Combinatorics and Homology}, Vol. {\bf 1} (C. Francisco, L. C. Klingler, S. Sather-Wagstaff and J. C. Vassilev, Eds.), De Gruyter, Berlin (2012), pp. 85--126.
 
\bibitem{Ng}
H.D. Nguyen,
\emph{Notes on the linearity defect and applications}.
Submitted (2015), available online at \newblock \verb|http://arxiv.org/abs/1411.0261|.

\bibitem{NgV}
H.D. Nguyen and T. Vu,
\emph{Linearity defects of powers are eventually constant}.
Submitted (2015), available online at \newblock \verb|http://arxiv.org/abs/1504.04853|.


\bibitem{NgV2}
H.D. Nguyen and T. Vu,
\emph{LinearityDefect. A Macaulay2 package for computing the linearity defect}. 
Available at \newblock \verb|http://www.math.unl.edu/~tvu5/research/LinearityDefect.m2|.

\bibitem{OkaYan}
R. Okazaki and K. Yanagawa,
\emph{Linearity defect of face rings}.
J. Algebra {\bf 314} (2007), 362--382.

\bibitem{PV}
I. Peeva and M. Velasco,
\emph{Frames and degenerations of monomial resolutions}.
Trans. Amer. Math. Soc. {\bf 363}, no. 4 (2011), 2029--2046.

\bibitem{R}
P. Roberts,
\emph{Homological invariants of modules over commutative rings}. 
Seminaire de Mathematiques Superieures {\bf 72}, Les Presses de l'Universite de Montreal (1980).

\bibitem{Ro}
T. R\"omer,
\emph{On minimal graded free resolutions}. 
Dissertation, Essen (2001).
 
\bibitem{Se}
L.M. \c{S}ega,
{\em On the linearity defect of the residue field}.
J. Algebra {\bf 384} (2013), 276--290.

\bibitem{SS}
J. Spinrad and R. Sritharan,
\emph{Algorithms for weakly triangulated graphs}.
Discrete Appl. Math. {\bf 59} (1995), 181--191.

\bibitem{V}
A. Van Tuyl,
\emph{A beginner's guide to edge and cover ideals}.
in \emph{Monomial ideals, computations and applications}, 63--94,
Lecture Notes in Math., {\bf 2083}, Springer, Heidelberg (2013).

\bibitem{Vi}
R.H. Villarreal,
\emph{Monomial Algebras}.
Monographs and Textbooks in Pure and Applied Mathematics, {\bf 238}. Marcel Dekker, Inc., New York (2001).

\bibitem{W}
R. Woodroofe,
\emph{Matchings, coverings, and Castelnuovo-Mumford regularity}. 
J. Commut. Algebra {\bf 6} (2014), no. 2, 287--304.


\end{thebibliography}
\end{document}